\documentclass[11pt,a4paper,reqno]{amsart}
\usepackage[T1]{fontenc}

\usepackage{amsmath,amsfonts,amssymb,mathtools,extarrows,mathrsfs}

\usepackage[hmargin=3cm,vmargin=3.5cm]{geometry}
\usepackage[dvipsnames]{xcolor}	
\usepackage[backref=page]{hyperref}
\usepackage{nameref,zref-xr}     

\usepackage{mathbbol}
\usepackage{cjhebrew}
\usepackage{pdftexcmds}

\usepackage{tikz}
\usepackage{float}
\usetikzlibrary{decorations.pathmorphing}

\hypersetup{
colorlinks=true, linktocpage=true, pdfstartpage=1, pdfstartview=FitV,breaklinks=true, pdfpagemode=UseNone, pageanchor=true, pdfpagemode=UseOutlines,plainpages=false, bookmarksnumbered, bookmarksopen=true, bookmarksopenlevel=1,hypertexnames=true, pdfhighlight=/O,urlcolor=webbrown, linkcolor=RoyalBlue, citecolor=ForestGreen}

\usepackage{enumerate}      

\usepackage{scalerel}         
\usepackage{stmaryrd}

\usepackage[all]{xy}

\usepackage[tickmarkheight=.5em,textwidth=0.7*\marginparwidth- 0pt,textsize=small]{todonotes}
\presetkeys%
    {todonotes}%
    {color=Apricot}{}%
\definecolor{pastelpurple}{RGB}{180,160,210}
\definecolor{pastelblue}{RGB}{150,180,210}
\definecolor{pastelgreen}{RGB}{180,220,160}
\definecolor{pastelgreen2}{RGB}{140,200,140}
\definecolor{pastelyellow}{RGB}{255,250,160}
\definecolor{pastelorange}{RGB}{255,200,140}
\definecolor{pastelgray}{RGB}{180,200,200}
\definecolor{pastelred}{RGB}{240,150,150}
\definecolor{pastelblack}{RGB}{80,80,80}
\definecolor{darkpastelorange}{RGB}{210,150,110}

\newtheorem{theorem}{Theorem}
\newtheorem{conjecture}{Conjecture}

\newtheorem{lemma}{Lemma}[section]
\newtheorem{proposition}[lemma]{Proposition}
\newtheorem{corollary}[lemma]{Corollary}

\theoremstyle{remark}

\newtheorem{remark}[lemma]{Remark}
\newtheorem{example}[lemma]{Example}

\theoremstyle{definition}

\newtheorem{definition}[lemma]{Definition}

\newtheorem{introthm}{Theorem}

\usepackage{color}

\def\Tr{\mathrm{Tr}\,}
\def\partitions{\mathscr{P}}
\def\symmgroup#1{\mathfrak{S}_{#1}}

\def\Q{\mathbb{Q}}
\def\Z{\mathbb{Z}}

\newcommand{\YoungDiagramFrob}[2][1]{%
\begin{tikzpicture}[scale=#1,baseline=(current bounding box.north)]

  \def\rows{#2}

  \foreach \len [count=\r from 1] in \rows {
    \foreach \c in {1,...,\len} {
      \draw (\c-1,-\r+1) rectangle (\c,-\r);
    }
  }

  \pgfmathsetmacro{\d}{0}
  \foreach \len [count=\r from 1] in \rows {
    \ifnum\len<\r\relax
    \else
      \xdef\d{\r}
    \fi
  }

  \ifnum\d>0
    \draw[thick] (0,0) -- (\d,-\d);
  \fi

  \foreach \len [count=\r from 1] in \rows {
    \ifnum\r>\d\relax
    \else
      \pgfmathsetmacro{\cx}{\r-0.5}
      \pgfmathsetmacro{\cy}{-\r+0.5}
      \pgfmathsetmacro{\xend}{\len}
      \draw[->,thick] (\cx,\cy) -- (\xend,\cy);
    \fi
  }

  \foreach \r in {1,...,\d} {
    \pgfmathsetmacro{\cx}{\r-0.5}
    \pgfmathsetmacro{\cy}{-\r+0.5}
    \pgfmathsetmacro{\last}{\r}
    \foreach \len [count=\k from 1] in \rows {
      \ifnum\len<\r\relax
      \else
        \xdef\last{\k}
      \fi
    }
    \pgfmathsetmacro{\yend}{-\last}
    \draw[->,thick] (\cx,\cy) -- (\cx,\yend);
  }

\end{tikzpicture}%
}

\numberwithin{equation}{section}

\begin{document}

\title{On the large genus of Hurwitz numbers}

 \author[D.~Accadia]{Davide Accadia}
 \address{D.~A.: Dipartimento di Matematica, Informatica e Geoscienze, Universit\`a degli studi di Trieste,
 	Via Weiss 2, 34128	Trieste, Italia \& Istituto Nazionale di Fisica Nucleare (INFN), Sezione di Trieste}
 \email{davide.accadia@phd.units.it}

\author[D.~Lewa\'nski]{Danilo Lewa\'nski}
\address{D.~L.: Dipartimento di Matematica, Informatica e Geoscienze, Universit\`a degli studi di Trieste,
	Via Weiss 2, 34128	Trieste, Italia \& Istituto Nazionale di Fisica Nucleare (INFN), Sezione di Trieste}
\email{danilo.lewanski@units.it} 

\author[G.~Ruzza]{Giulio Ruzza}
\address{G.~R.: CEMS.UL, Departamento de Ci\^encias Matemáticas, Faculdade de Ci\^encias da Universidade de Lisboa, Campo Grande Edifício C6, 1749-016, Lisboa, Portugal}
\email{gruzza@fc.ul.pt} 
	
\begin{abstract} 
Hurwitz theory provides a large variety of enumerative problems related to algebraic geometry, mathematical physics, and combinatorics. 
We give a general framework to approach the large genus asymptotics of Hurwitz theory using only elementary methods and apply it to several types of Hurwitz numbers: single, double, or with an arbitrary numbers of fixed ramifications; simple and/or including completed cycles type of ramification and/or finitely many blocks of weakly monotone and/or strictly monotone types of ramifications. These, to the best of our knowledge, cover most of the Hurwitz numbers studied, and include for instance correlators of the HCIZ matrix model, Grothendieck dessins d'enfant, weighted Hurwitz numbers, and Gromov--Witten invariants of the Riemann sphere. We also apply our method to $b$-content Hurwitz numbers. As a specialisation, we recover some previously known about the large genus asymptotics of Hurwitz theory, namely classical results by Hurwitz and recent results of Do--He--Robertson, C.~Yang, and results connected to recent work of X.~Li.
\end{abstract}

\maketitle
\tableofcontents

\section{Introduction}
\label{sec:intro}

Hurwitz theory was introduced by A.~Hurwitz in 1891 \cite{Hur} as the enumeration of genus $g$, degree $d$, compact ramified coverings of the Riemann sphere with only simple ramifications over fixed points, which by the Riemann--Hurwitz equation must be in number $2g - 2 + 2d$. Classical Hurwitz numbers are the result of this enumeration. Since their initial introduction a rich variety of other Hurwitz numbers have been defined, and a striking number of fundamental results involving relating them to integrable hierarchies of PDEs, the algebraic geometry of moduli spaces of curves, a number of new and known combinatorial recursions, a deep connection with Gromov--Witten theory, the topological recursion in the sense of B.~Eynard and N.~Orantin, and much more, has been discovered. 
In what follows we define the Hurwitz numbers that will concern this paper and we give a very brief summary of the links with other branches of mathematics and physics as well as what is known about their large genus asymptotics and the main results of the paper.

\subsection{Preliminaries on Hurwitz theory}
\label{sec:Hur}

For $d\in\Z_{\geq 0}$ let $\partitions_d$ be the the set of all (integer) partitions of $d$ and $\partitions=\bigcup_{d=0}^\infty\partitions_d$ the set of all partitions. We will use the standard notations $|\lambda|=d$ or $\lambda\vdash d$ to say that $\lambda\in\partitions_d$ and we will denote the length of a partition~$\lambda$ by $\ell(\lambda)$.

Let $\symmgroup d$ be the symmetric group on $d$ elements, $\Q [\symmgroup d] $ be its group algebra over the rationals, namely
\begin{equation}
\Q [\symmgroup d] = \biggl\lbrace  \sum^{s}_{i=1} a_i \sigma_i \; \bigg| \; a_i \in \Q, \, \sigma_i \in \symmgroup d \biggr\rbrace,
\end{equation}
and $Z(\Q[\symmgroup d])$ be its center.
For any $\sigma \in \symmgroup d$, let the cycle type of $\sigma$ be the partition $\mathrm{cyc}(\sigma)\vdash d$ whose parts are the lengths the disjoint cycles of $\sigma$.
For example, if $\sigma = (3 6 5 7)(2)(1 4 8) \in \symmgroup 8$, $\mathrm{cyc}(\sigma) = (4,3,1) \in\partitions_8$. 
For a given $\mu\vdash d$, the set
\begin{equation}
\mathfrak{C}_\mu =\bigl\lbrace \sigma\in\symmgroup d\;\big|\;\mathrm{cyc}(\sigma)=\mu\bigr\rbrace
\end{equation}
is a conjugacy class in $\symmgroup d$, and so the elements
\begin{equation}
\label{eq:Cmu}
\mathcal{C}_\mu = \sum_{\sigma \in\mathfrak{C}_\mu}\sigma
\end{equation}
(a priori, elements of $\Q[\symmgroup d]$) actually form a linear basis of $Z(\Q[\symmgroup d])$ (indexed by $\mu\vdash d$).
For more details on the group algebra and on the other notions we are about to present see, for instance, the book \cite{CM} and Appendix~\ref{sec:Jucys}.
Another linear basis of the center $Z(\Q[\symmgroup d])$, also indexed by $\partitions_d$, is formed by the idempotent elements $\mathcal{F}_{\lambda}$ ($\lambda\vdash d$), satisfying
\begin{equation}
\label{eq:idempotent}
\mathcal{F}_{\lambda} \,\mathcal{F}_{\eta} \,=\, \delta_{\lambda, \eta}\,\mathcal{F}_{\lambda}
\qquad (\lambda,\eta\vdash d).
\end{equation}
The change of basis between $\left\{ \mathcal{F}_{\lambda} \right\}_{\lambda \vdash d}$ and $\left\{ \mathcal{C}_{\mu} \right\}_{\mu \vdash d}$ follows from orthogonality of characters and is given by:
\begin{equation}\label{eq:change:bases:C:F}
    \mathcal{F}_\lambda \,=\,\frac{\dim(\lambda)}{d!} \sum_{\mu \vdash d} \chi_{\lambda}(\mathcal{C}_\mu)\, \mathcal{C}_{\mu}, \qquad \qquad \mathcal{C}_{\mu} \,=\, |\mathfrak{C}_{\mu}| \sum_{\lambda \vdash d} \frac{\chi_{\lambda}(\mathcal{C}_\mu)}{\dim(\lambda)} \,\mathcal{F}_{\lambda}.
\end{equation}
Here, $\chi_\lambda(\mathcal{C}_\mu)\in\Q$ is the value of the character of the irreducible representation of $\symmgroup d$ associated with the partition~$\lambda\vdash d$ on any permutation in the conjugacy class $\mathfrak{C}_\mu$ and $\dim\lambda=\chi_\lambda(\mathrm{id})$ is the dimension of the same irreducible representation.

Moreover, we need the Jucys--Murphy elements \cite{Jucys, Murphy} which are defined as formal sums of transpositions:
\begin{equation}
\mathcal{J}_1 \coloneqq 0 \in \Q[\symmgroup d], 
\qquad
\mathcal{J}_k \coloneqq \sum_{i = 1}^{k-1} (i \; k) \in \Q[\symmgroup d], \quad k = 2, \dots, d,
\end{equation}
where $(a \; b) \in \symmgroup d$ is the transposition exchanging $a$ and $b$.
They satisfy many interesting properties, which we summarise in Appendix~\ref{sec:Jucys}.

\subsection{Different types of Hurwitz numbers}

We now recall the definition of different types of Hurwitz numbers, divided into two main families: the first (Sections~\ref{sec:defclassicalHN}--\ref{sec:defcompletedcyclesHN}) involves \textit{completed cycles} ramifications, after the work of A. Okounkov and R. Pandharipande, and the second (Sections~\ref{sec:defhypergeometricHN}--\ref{sec:defmonotoneHN}) involves rational generating series evaluated at Jucys--Murphy elements, in which case we corresponding Hurwitz numbers are said to be \textit{hypergeometric}, after the work of  M.~Guay-Paquet, J.~Harnad, and A.~Orlov~\cite{GPH,HO}.

\subsubsection{Classical Hurwitz numbers}\label{sec:defclassicalHN}
These are the ones originally considered by A.~Hurwitz \cite{Hur} and they enumerate genus $g$ degree $d$ possibly disconnected ramified covers of the Riemann sphere with $r$ simple ramifications only, with a weight equal to the inverse of the order of the automorphism group of the cover.
By definition, a simple ramification is one with ramification profile $\tau=(2,1,1,\dots)\vdash d$ (only two sheets merge together over the ramification point), hence the classical Hurwitz numbers can be defined by 
\begin{equation}\label{eq:classical:H}
    H^{\bullet}_{r,d} = \frac{1}{d!} [\mathcal{C}_{(1^d)}].\left(\mathcal{C}_{\tau}\right)^r \in \Q_{\geq 0}.
\end{equation}
In this formula, the element $\mathcal{C}_{\tau}$ is, according to the general formula~\eqref{eq:Cmu}, the formal sum of all transpositions and the operator $[\mathcal{C}_{(1^d)}]$ extracts the coefficient of $\mathcal{C}_{(1^d)}$ after expanding the product $(\mathcal{C}_{\tau})^r$ in the linear basis $\bigl\lbrace \mathcal{C}_\mu\bigr\rbrace_{\mu\vdash d}$ of $Z(\mathbb{Q}[\symmgroup d])$. 
We note that the Riemann--Hurwitz formula implies the following relation of parameters
\begin{equation}
    r = 2g - 2 + 2d,
\end{equation}
and that the prefactor $1/d!$ is included in~\eqref{eq:classical:H} to account for all possible ways of labelling the $d$ cover sheets.
That formula~\eqref{eq:classical:H} corresponds to the above-mentioned enumeration problem follows from Riemann's existence theorem (note that the product of all monodromies over ramification points must give the trivial monodromy, as the Riemann sphere is simply connected); all details are, for example, in the book~\cite{CM}.

The \textit{connected} classical Hurwitz numbers $H^{\circ}_{r,d}$ are obtained from possibly disconnected ones $H^{\bullet}_{r,d}$ by the inclusion-exclusion formula or, equivalently, by the identity of generating series
\begin{equation}
\label{eq:logdisc}
    \sum_{r\geq 0}\sum_{d\geq 1}H^{\circ}_{r,d}\frac{x^r}{r!}y^d 
    = 
    \log  \left( 1\,+\,\sum_{r\geq 0}\sum_{d\geq 1} H^{\bullet}_{r,d}\frac{x^r}{r!}y^d\right).
\end{equation}

\subsubsection{Single, double, and $N$-tuple Hurwitz numbers}\label{sec:defmultHN}

Classical Hurwitz numbers can be generalised to single Hurwitz numbers, corresponding to a similar orbifold count of ramified covers (of genus $g$ and degree $d$) of the Riemann sphere in which one distinguished ramification point (say, the point~$0$ of the Riemann sphere) has an assigned ramification profile $\mu\vdash d$ while the remaining $r$ ramification points are simple. 
Accordingly, the single Hurwitz numbers can be defined by
\begin{equation}\label{eq:classical:H:disc}
    H^{\bullet}_{r}(\mu) = \frac{1}{d!
    \,|\mathfrak{C}_\mu|} [\mathcal{C}_{(1^d)}]. \mathcal{C}_{\mu}\left(\mathcal{C}_{\tau}\right)^r \in \Q_{\geq 0}.
\end{equation}
This time, the Riemann--Hurwitz formula reads
\begin{equation}
    r = 2g - 2 + d + \ell(\mu).
\end{equation}
Adding a second distinguished ramification point (say, the point $\infty$ of the Riemann sphere) with an assigned ramification profile $\nu \vdash d$ yields the \textit{double Hurwitz numbers}:
\begin{equation}\label{eq:classical:H:double}
    H^{\bullet}_{r}(\mu, \nu) = \frac{1}{d!
   \,|\mathfrak{C}_\mu|\,|\mathfrak{C}_\nu|
    } [\mathcal{C}_{(1^d)}]. \mathcal{C}_{\mu}\mathcal{C}_{\nu} \left(\mathcal{C}_{\tau}\right)^r \in \Q_{\geq 0},
\end{equation}
with
\begin{equation}
    r = 2g - 2 + \ell(\nu) + \ell(\mu).
\end{equation}
Generalising, given an arbitrary number $N \in \Z_{\geq 0}$ of ramification profiles $\mu^{(1)}, \dots, \mu^{(N)}\vdash d$ leads to $N$-tuple Hurwitz numbers:
\begin{equation}\label{eq:classical:H:N}
    H^{\bullet}_{r}(\mu^{(1)}, \dots, \mu^{(N)}) = \frac{1}{d!
    \,\prod_{j=1}^N\left|\mathfrak{C}_{\mu^{(j)}}\right|
    } [\mathcal{C}_{(1^d)}]. \left(\prod_{j=1}^N \mathcal{C}_{\mu^{(j)}} \right)\left(\mathcal{C}_{\tau}\right)^r \in \Q_{\geq 0},
\end{equation}
with the Riemann--Hurwitz formula
\begin{equation}
    r = 2g - 2 - d(N - 2) + \sum_{j=1}^N \ell(\mu^{(j)}).
\end{equation}
In all these cases, connected Hurwitz numbers $H^{\circ}_{r}$ can be defined from possibly disconnected ones $H^{\bullet}_{r}$ by inclusion-exclusion formula or, equivalently, by taking the logarithm of an appropriate generating function of disconnected Hurwitz numbers, as in~\eqref{eq:logdisc}.

\subsubsection{$(s+1)$-completed cycles ramifications}\label{sec:defcompletedcyclesHN}

For $s \in \Z_{\geq 1}$, the $(s+1)$-completed cycles are elements
\begin{equation}
    \overline{\mathcal{C}}_{s+1} \in \Q[\symmgroup d], \qquad \overline{\mathcal{C}}_{s+1} = \mathcal{C}_{(s+1, 1^{d-s-1})} + \text {l.o.t.},
\end{equation}
defined as a specific linear combination of elements $ \mathcal{C}_{\mu}$ where the \emph{leading term} represents a ramification with $(s+1)$ sheets of the cover merging together over a point of ramification index $s$, plus a tail of \emph{lower order terms}, i.e. involving strictly lower ramification index, whose coefficients are essentially given by the inverse of the hyperbolic sine function. These specific elements have proven to be the ones corresponding to the Gromov--Witten insertions in the stationary sector of the Riemann sphere, eventually leading to the resolution of the whole Gromov--Witten theory of smooth algebraic curves in the trilogy~\cite{OP1, OP2, OP3} by A.~Okounkov and R.~Pandharipande, to which we refer for more details.
Geometrically, they still enumerate certain linear combinations of Riemann surface covers, see for instance \cite{LPSZ, SSZ} and references within for more details.

Hurwitz numbers with $(s+1)$-completed cycles can be defined as
\begin{equation}\label{eq:classical:H:N:completed}
    H_{r}^{\bullet, s} (\mu^{(1)}, \dots, \mu^{(N)}) = \frac{1}{d!
    \,\prod_{j=1}^N\left|\mathfrak{C}_{\mu^{(j)}}\right|
    } [\mathcal{C}_{(1^d)}]. \prod_{j=1}^N \mathcal{C}_{\mu^{(j)}}(\overline{\mathcal{C}}_{s})^r \in \Q_{\geq 0},
\end{equation}
with \begin{equation}
    r s = 2g - 2 - d(N - 2) + \sum_{j=1}^N \ell(\mu^{(j)}).
\end{equation}
When $N=1$ or $N=2$ we refer to these numbers as single or double (respectively) $(s+1)$-completed cycles Hurwitz numbers. \\
When $s=1$, we have
\begin{equation}
    \overline{\mathcal{C}}_{2} = \mathcal{C}_{\tau} \in Z(\Q[\symmgroup d]),
\end{equation}
without any \textit{completion} needed, hence in this case we fall back to classical Hurwitz numbers of Section~\ref{sec:defclassicalHN}, corresponding to simple ramifications only.
Also, different parameters $s$ can be mixed in the same definition of Hurwitz numbers, giving the full connection with the stationary sector of the Gromov--Witten invariants of $\mathbb{CP}^1$ \cite{OP1}. \\
Finally, the connected version of these Hurwitz numbers, $H^{\circ, s}_{r}$, can be defined from possibly disconnected ones $H^{\bullet, s}_{r}$ by inclusion-exclusion formula.

\subsubsection{Hypergeometric-type Hurwitz numbers}\label{sec:defhypergeometricHN}
Over the last two decades, starting with the seminal work of A.~Okounkov~\cite{Ok}, a particularly fruitful perspective has emerged through the realization that many families of Hurwitz numbers are encoded by $\tau$-functions of integrable hierarchies, most notably the KP and 2D Toda hierarchies, see \cite{DHR, DYZ, K, CYang}.

With notation as before and for a given $G(z)\in 1+z\Q[[z]]$, one can define the \textit{hypergeometric Hurwitz numbers}
\begin{equation}
\label{eq:defHurwitz}
H_r^{\bullet, G}(\mu^{(1)}, \dots, \mu^{(N)})=\frac{1}{d!
\,\prod_{j=1}^N\left|\mathfrak{C}_{\mu^{(j)}}\right|
}\bigl[\mathcal{C}_{(1^d)}z^{r} \bigr]. \prod_{j=1}^N \mathcal{C}_{\mu^{(j)}}\prod_{i=1}^d G(z \mathcal{J}_i),
\end{equation}
where the power series $G$ gets formally evaluated at $z$ times Jucys--Murphy elements, all factors are then expanded into the group algebra and finally the operator $\bigl[\mathcal{C}_{(1^d)}z^{r} \bigr]$ counts how many of these products realise the identity permutation after the coefficient of the monomial $z^r$ is extracted. \\
In the work of M.~Guay-Paquet, J.~Harnad, and A.~Orlov \cite{GPH,HO}, hypergeometric $\tau$-functions were shown to provide a unifying framework for a broad class of weighted Hurwitz numbers, including monotone, strictly monotone, and Grothendieck dessins d'enfant enumerations, via content-product deformations associated with symmetric group representations. \\
In particular, we will work with an (almost) arbitrary rational function, which up to constant multiplication can be written as
\begin{equation}\label{eq:G}
G(z) \coloneqq
\frac{(1+u_1 z)\cdots(1+u_L z)}
{(1-v_1 z)\cdots(1-v_M z)}
\end{equation}
for some integers $M > 0$, $L \geq 0$, and some formal variables $u_1, \dots, u_L$, and $v_1, \dots, v_M$. These Hurwitz numbers again enumerate degree $d$ and genus $g$ Riemann surfaces coverings of the Riemann sphere, with $N$ fixed ramifications $\mu^{(j)}$ over fixed points on the base, and only simple ramifications elsewhere, but organised in a very specific way. The collection of transpositions is partitioned in $L+M$ blocks: in the first $L$ blocks the transpositions are strictly monotone while in the last $M$ blocks the transpositions are weakly monotone (or just monotone for short). In practice, if $P$ simple ramifications arise from a cover, we can always write them as transpositions,
\begin{equation}
    (a_1 \; b_1), \dots, (a_P \; b_P), \qquad \qquad a_i < b_i, \;\; i = 1, \dots, P,
\end{equation}
then, a weakly (strictly) monotone cover respects the following further constraint to be imposed on the transpositions 
\begin{equation}
    b_i \leq b_{i+1}, \; \text{(resp. } b_i < b_{i+1} ), \; \text{ for } \; i = 1, \dots, P-1.
\end{equation}
The monotone condition is reset at the end of each block of simple ramifications. The number of simple ramifications extracted from each block is determined by the exponent of the corresponding variable, for instance, the coefficient of $u_i^5$ enumerates all those covers with exactly $5$ simple ramifications from the $i$-th strictly monotone block of ramifications. For more details see \cite{ALS, HO, OSnumbers} and references therein. Again, the connected numbers are defined from their disconnected version by means of the inclusion--exclusion formula.

\subsubsection{Strictly and weakly monotone Hurwitz numbers}\label{sec:defmonotoneHN}
For strictly monotone Hurwitz numbers, i.e., one single strictly monotone block, let $G(z)=1+z$. In this case the Hurwitz numbers enumerate factorizations of the identity in~$\symmgroup d$ in transpositions $(a_i,b_i)$ satisfying $a_i<b_i$ and $b_i<b_{i+1}$. 
They are related to the correlators of the Gaussian Unitary Ensemble matrix model and to Grothendieck dessins d'enfants, see, e.g., \cite{ALS, BorotGarciaFailde} and references therein.

For monotone Hurwitz numbers, i.e., one single weakly monotone block, let $G(z)=(1-z)^{-1}$. In this case the Hurwitz numbers enumerate factorizations of the identity in~$\symmgroup d$ of transpositions $(a_i,b_i)$ satisfying $a_i<b_i$ and $b_i \leq b_{i+1}$. They are related to the correlators of the Harish-Chandra--Itzykson--Zuber matrix model (see section \ref{sec:HCIZ}). For more details on how Jucys--Murphy elements encode the monotone conditions on the ramifications and the correspondence between strictly monotone Hurwitz numbers and Grothendieck dessins see \cite{D}.
\begin{remark} Notice that
    $
        \sum_{i=1}^d \mathcal{J}_i = \mathcal{C}_\tau,
    $
    hence if $G(z)=\exp(z)$ we have \begin{equation}
    \prod_{i=1}^d G(z \mathcal{J}_i) = \exp\bigl(z \,\mathcal{C}_\tau\bigr),
    \end{equation}
    hence, classical Hurwitz numbers are an example of hypergeometric Hurwitz numbers.
\end{remark}
\begin{remark} The hypergeometric Hurwitz numbers corresponding to rational weight generating functions of the form $G(z)=(1+az)(1+z)$, $G(z)=\frac{1+az}{1-z}$, and $G(z)=\frac{(1+az)(1+z)}{1-bz}$ occur naturally in the correlators of the Laguerre, inverse Laguerre, and Jacobi Unitary Ensembles of Hermitian random matrices, respectively. For more details see, ~\cite{Gisonni1,Gisonni2}, and references therein.
\end{remark}

\subsubsection{Remarks and possible future directions} 
It is natural to compose different types of ramifications, including both completed cycles of different order and a generic rational or analytic function $G$ evaluated at the Jucys--Murphy elements.
Moreover, it is possible to enumerate coverings with slightly different weights that have certain geometric meaning. 
To name a few:
\begin{itemize}
\item \textit{Spin} Hurwitz numbers arise when additionally weighting coverings by a sign given by an associated spin structure lifted up from the basis \cite{GKL, GKLS}.
\item $b$-Hurwitz numbers introduce a complex parameter $b$ as a ``measure of non-orientability'', see, e.g., \cite{BCD1,BCD2,CD,FHKM,R}.
\item Different algebraic groups can replace  the symmetric group permuting the cover sheets \cite{Karev},
\item \textit{Real} Hurwitz numbers correspond to covers which are invariant with respect to a given complex conjugation \cite{Guidoni}.
\end{itemize}

One could also try to study combinations of the different ramification conditions we have seen in different context, like the spin, $b$-, real, etc. context, for the ramifications that make sense. This opens up a number of unexplored possibilities, many of which may now be at reach for what concerns the calculation of large genus asymptotics, by adapting the methods of the current work.

\subsection{Hurwitz numbers in terms of representation theory}

It is useful to express all of the above Hurwitz numbers using representation theory.
More precisely, we now express them explicitly as a sum over irreducible representations of evaluations of characters and other combinatorial ingredients arising from the chosen group of reference which, in this case, is the symmetric group $\symmgroup d$.
In the simple case (simple ramifications without any monotone condition) and arbitrary $N$ this formula is classical and known as Burnside formula.
The general strategy to obtain such formula is simply to apply the base change between the bases $\left\{ \mathcal{F}_{\lambda} \right\}_{\lambda \vdash d}$ and $\left\{ \mathcal{C}_{\mu} \right\}_{\mu \vdash d}$ given in~\eqref{eq:change:bases:C:F} and exploit the idempotency of the $\mathcal{F}_{\lambda}$-basis, see~\eqref{eq:idempotent}, simplifying the formula considerably. The hard part is to decompose the remaining factors of the formula, namely $G$ evaluated at the Jucys--Murphy elements or the completed cycles, into the idempotent basis. Let us now show how to achieve this formula for the cases we are interested in.

\subsubsection{Completed cycles Hurwitz numbers in terms of representation theory}
By applying the change of basis \eqref{eq:change:bases:C:F} on the $\mathcal{C}_{\mu}$ and using the shifted Frobenius coordinates defined in the Appendix \ref{sec:Jucys} one obtains the formula:
\begin{equation}
\begin{split}
H_r^{\bullet, s+1} &\bigl(\mu^{(1)},\dots,\mu^{(N)}\bigr) 
=
\\
&= \sum_{\lambda\vdash d} \biggl(\frac{\dim\lambda}{d!}\biggr)^2\biggl(\prod_{i=1}^N\frac{\chi_{\lambda}(\mathcal{C}_{\mu^{(i)}})}{\dim\lambda}\biggr) \biggl(\frac{1}{s+1}\sum_{i=1}^{R(\lambda)} \bigl( (a_i')^{s+1} - (-b_i')^{s+1} \bigr)\biggr)^{r},
\end{split}
\end{equation}
where $a_i', b_i'$ ($1\leq i\leq R(\lambda)$) are the shifted Frobenius coordinates of~$\lambda$, see Section~\ref{sec:Shifted:Frob}.

\subsubsection{Hypergeometric Hurwitz numbers in terms of representation theory}

By applying the change of basis \eqref{eq:change:bases:C:F} on the $\mathcal{C}_{\mu}$ and using equation \eqref{eq:Jucys:eigen} as well as the fact that $\mathcal{F}_\lambda$ are idempotent, see~\eqref{eq:idempotent}, one obtains the formula 
\begin{equation}
\label{eq:Frobenius}
H_r^{\bullet, G}\bigl(\mu^{(1)},\dots,\mu^{(N)}\bigr)
=
[z^r].\sum_{\lambda\vdash d} \biggl(\frac{\dim\lambda}{d!}\biggr)^2\biggl(\prod_{i=1}^N\frac{\chi_{\lambda}(\mathcal{C}_{\mu^{(i)}})}{\dim\lambda}\biggr)
\prod_{(i,j)\in Y_\lambda} 
G\bigl(z(j-i)\bigr).
\end{equation}
Here, $Y_\lambda$ is the diagram of $\lambda$, namely
\begin{equation}
    \label{eq:Ylambda}
    Y_\lambda = \bigl\lbrace (i,j)\in(\Z_{\geq 1})^2:\, 1\leq i\leq \ell(\lambda),\,1\leq j\leq \lambda_i\bigr\rbrace.
\end{equation}
Note that $G$ gets evaluated at the contents $j-i$ of the Young diagram of $\lambda$ and the coefficient of $z^r$ is collected after the taking the Taylor series at $z=0$. 
We stress that all sums and products involved in the formula are finite, hence the expression to which $[z^r]$ is applied is again a rational function of $z$ and of the formal variables $u_i$ and $v_j$ involved in the definition of $G(z)$ in \eqref{eq:G}.
\subsection{Large genus asymptotics of Hurwitz numbers and main results}

As the result of the many definitions of Hurwitz numbers, often motivated by different fields, Hurwitz theory became, over the last $30$ years, a rich topic occupying nowadays a central position at the interface of enumerative algebraic geometry, intersection theory on the moduli spaces of curves, Gromov--Witten theory, Mirror Symmetry, topology, the combinatorics of maps and graphs of Riemann surfaces, representation theory, symmetric functions, random matrix models, free probability, topological recursion and integrable hierarchies of integrable partial differential equations. 

The main issue of enumerative geometry is that, often,  problems that are interesting are complex enough that a closed formulae (e.g. without summing over partitions or permutations) is completely out of reach. There are two main approaches to deal with the lack of closed formulae: the first one is to construct efficient recursions, the second is control the asymptotics for large parameters. \\
The first approach is usually as precise as it is slow. For Hurwitz numbers there exist a variety of recursions arising from KP or 2D Toda, Cut-and-join equations of different type, topological recursion in the sense of B.~Eynard and N.~Orantin, computations via Fock space operators, and more, but these recursions do not reach very far in the genus parameter. \\
The second approach can be extremely powerful for large genus, although it discards a certain amount of information, and it can be quite a challenging result to establish.

There has been much work lately about establishing the large genus asymptotics for enumerative problems arising from algebraic geometry and theoretical physics, not too far from Hurwitz theory. This work has been carried out adapting a variety of techniques, from resurgence to probability, from topological recursion to integrability, which are interesting methods in themselves. For instance, the large genus of Witten--Kontsevich numbers have been studied (\cite{Agg,Psilarge, GY, MZ} and references within), the Brezin--Gross--Witten model \cite{GNYZ}, the Masur--Veech volumes \cite{theseven, CMS, YZZ} as well as the large genus of the Jackiw--Teitelboim gravity~\cite{JTlarge}. 

Some of these advanced techniques have been applied to Hurwitz theory, but without success. The point of the current work is that, surprisingly, simpler and classical techniques seem to work for Hurwitz theory. These techniques just involve their known expression in terms of representation theory and the well--known fact that the asymptotics of a Taylor series is controlled by the closest singularity. Under this consideration the representation theoretic expression greatly simplifies and the sum over partitions drops, leaving us with a controllable expression.

We now illustrate what was known about the large genus asymptotics of Hurwitz numbers, the recent developments by N.~Do, J.~He, and H.~Robertson \cite{DHR}, as well as by C.~Yang \cite{CYang}, and the results of the current paper. 

For simple Hurwitz numbers and $N=0$, Hurwitz proved that
\begin{theorem}[\cite{Hur}]
\begin{equation}
    H_r^{\circ} \sim \frac{2}{d!^2} \binom{d}{2}^{2g - 2 + 2d}
    \quad \text{ as } g \to \infty.
\end{equation}
\end{theorem}

Recently, N.~Do, J.~He, and H.~Robertson extended this result to~$N=1$.
\begin{theorem}[\cite{DHR}]\label{thm:DHR}
\begin{equation}
    H_r^{\circ}(\mu) \sim \frac{2}{d!^2} \binom{d}{2}^{2g - 2 + d + \ell(\mu)}
    \quad \text{ as } g \to \infty.
\end{equation}
\end{theorem}

This work provides a generalisation in two different directions: from $s=1$ to arbitrary $s$, from $N=1$ to arbitrary $N$. Moreover, connected Hurwitz numbers are shown to be asymptotically equivalent to possibly disconnected Hurwitz numbers for large genus covers.

The first result provides a generalisation to completed cycles.

\begin{introthm}[see Theorem \ref{thm:large:g:completed}] We have 
\begin{equation}
    \begin{split}
        H_r^{\circ, s+1} &\bigl(\mu^{(1)},\dots,\mu^{(N)}\bigr) \sim H_r^{\bullet, s+1} \bigl(\mu^{(1)},\dots,\mu^{(N)}\bigr) \sim \frac{2}{d!^2} (M_{d,s})^{r} 
    \end{split}
\end{equation}
    as $g \to \infty$, for $sr = 2g - 2 - d(N-2) + \sum_{i=1}^N \ell(\mu^{(i)})$ and $$M_{d,s} = \frac{1}{s+1} \left[\left( d-\frac{1}{2} \right)^{s+1} - \left( -\frac{1}{2} \right)^{s+1} + (1 - 2^{-s-1})\zeta(-s-1) \right ].$$
    with $\zeta$ the Riemann zeta function.
\end{introthm}

For $s=1$ this result appears in \cite{XiangLi}, and for $(s+1)$-\emph{not completed cycles} in \cite{XiangLi_r}, derived with similar techniques, independently at the same time. Notice that $M_{d,1} = \binom{d}{2}.$ On the other hand, for $s=1$ \cite{DHR} provides the structure theorem
\begin{equation}
    H_r^{\circ}(\mu) = \frac{2}{d!^2} \sum_{1 \leq m \leq \binom{d}{2}} C^{\circ}(\mu, m) m^{2g - 2 + d + \ell(\mu)}.
\end{equation}
The leading coefficient is there computed as
\begin{equation}
    C^{\circ}\left(\mu, \binom{d}{2}\right) = C^{\circ}\left(\mu, M_{d,1} \right) = 1,
\end{equation}
clearly implying Theorem~\ref{thm:DHR} in the same paper. Furthermore, the following coefficient gap is conjectured.
\begin{conjecture}[{\cite[Conjecture 4.1]{DHR}}]\label{conj:DHR}
    \begin{equation}
    C^{\circ}\left(\mu, m \right) = 0, \qquad \text{ for } \binom{d-1}{2} < m < \binom{d}{2}.
\end{equation}
\end{conjecture}
The conjecture is solved positively in~\cite{CYang}.

\begin{theorem}[{\cite[Theorem 1.3]{CYang}}]
    Conjecture \ref{conj:DHR} holds true and moreover
    \begin{equation}
        C^{\circ} \left(\mu, \binom{d-1}{2}\right) = -d\cdot m_1(\mu),
    \end{equation}
    where $m_i(\mu)$ is the multiplicity of parts equal to $i$ inside the partition $\mu$.
\end{theorem}

This work provides a similar result for possibly disconnected Hurwitz numbers, generalised in three different directions: from $s=1$ to arbitrary $s$, from $N=1$ to arbitrary $N$, and from the computation of the coefficient gap and the first subleading coefficient to the computation in closed form of all coefficients in terms of polynomials of the partitions multiplicities.

\begin{introthm}[see Theorem \ref{thm:structure:disconnected}]
\begin{equation}
    \begin{split}
        H_r^{s+1, \bullet}&(\mu^{(1)}, \dots, \mu^{(N)}) = \frac{2}{d!^{2}}
        \sum_{
        m_i \in \overline{f}_s(\mathcal{P}_{d_i})} 
        C_s^{\bullet}(\vec{\mu}, m) \cdot m^{r}.
    \end{split}
\end{equation}
    In particular leading coefficient reads
    \begin{equation}
        C_s^{\bullet}\left(\vec{\mu}, M_{d,s} \right) = 1,
    \end{equation}
    and we have the coefficient gap
    \begin{equation}
        C_s^{\bullet}(\vec{\mu}, m) = 0, \quad \text{ for} \quad   \overline{f}_{s+1}((d-1)) - 1 < m < \overline{f}_{s+1}((d)).
    \end{equation}
\end{introthm}
Notice that the vanishing gap range for possibly disconnected numbers is one unit longer, when specialised to $s=1$, with respect to the known gap range for connected numbers. Again for $s=1$ this result appears in \cite{XiangLi}, and for $(s+1)$-\emph{not completed cycles} in \cite{XiangLi_r}, derived with similar techniques, independently at the same time.

By means of the inclusion-exclusion formula, the result can be transferred from possibly disconnected to connected numbers. This is achieved in Proposition \ref{thm:structure:connected}. Since the inclusion-exclusion formula expresses connected numbers as \emph{polynomials} in the possibly disconnected numbers, the resulting structure statement is more complex and does not immediately specialise to the connected structure theorem above for $s=1$.

For monotone Hurwitz numbers the work \cite{DHR} also proposed a conjectural structure theorem, that has been established in \cite{CYang} by analysing the KP integrability structure (after correcting the conjecture). This implies the following large genus asymptotics.

\begin{theorem}[{\cite[Proposition 1.2]{CYang}}] For $M=1, L=0$ and $N=1$ we have 
    \begin{equation}
        H^{G, \circ}_r(\mu) \sim \frac{2}{d!^2} \frac{(d-1)^{d-2}}{(d-2)!} (d-1)^{r} + O((d-2)^{r})
    \end{equation}
    for $r = 2g - 2 + d + \ell(\mu).$
\end{theorem}

This work generalises the above in four different directions: from $M=1$ to arbitrary $M$, from $L=0$ to arbitrary $L$, from $N=1$ to arbitrary $N$, from the number of simple ramifications of a single weakly monotone block going to infinity with the genus, i.e. $K=1$, to an arbitrary number $K$ of blocks.

\begin{introthm}[see theorem \ref{thm:large:g:monotone:K}]
\begin{align*}
\Big[
&
\prod_{i=1}^L u_i^{a_i}
\prod_{j=1}^{M-K} v_j^{b_j}
\Big]
 H_r^{G, \circ}(\mu^{(1)}, \dots, \mu^{(N)})
\sim
\\
&\sim
2 \frac{r^{K-1}}{(K-1)!}\frac{(d-1)^{(d-2)K-\sum a_i-\sum b_j+r}}{d!^2 (d-2)!^K}
\prod_{i=1}^L \biggl[\begin{matrix} d \\ d - a_i \end{matrix}\biggr]
\prod_{j=1}^{M-K} \biggl\{\begin{matrix} b_j + d - 1 \\ d-1 \end{matrix}\biggr\},
\end{align*}
as $g\to +\infty$ with 
$
r = 2g - 2 - d(N-2) + \sum_{j=1}^N \ell(\mu^{(j)}).
$
\end{introthm}

We conclude this work with applications to Gromov--Witten theory (see section \ref{sec:GW}), orbifold Hurwitz numbers (see section \ref{sec:orbi}), the coefficients of HCIZ random matrix model (see section \ref{sec:HCIZ}) and $b$-Hurwitz numbers (see section \ref{sec:b}).

\subsection{Organisation of the paper}
The paper is organized as follows. We review the necessary background on Hurwitz numbers in section \ref{sec:Hur}, and on Jucys–Murphy elements, symmetric polynomials and Stirling numbers in section \ref{sec:Jucys}.

There are two families of Hurwitz numbers under considerations: the one of hypergeometric type and the ones with completed cycles.

In section \ref{sec:hypergeometric} we then analyse the hypergeometric family genus dependence using the Frobenius formula and deriving asymptotics in the large genus regime via the analysis of their closest poles, their order and their coefficients. Finally, we discuss specialisation to strictly monotone Hurwitz numbers and Grothendieck dessins, we generalise the main statement to multiple blocks of weakly monotone ramifications growing together with the genus, and we retrieve a recent result of C.~Yang for a single block of weakly monotone Hurwitz numbers \cite{CYang}.

In section \ref{sec:completed:cycles} we analyse the completed cycles family genus dependence using shifted Frobenius coordinates of a partition and deriving asymptotics in the large genus regime via the analysis of the biggest eigenfunction of the corresponding Fock space operators. This paves the way to the large genus asymptotics of Gromov--Witten invariants of the Riemann sphere via the Gromov--Witten / Hurwitz correspondence of A.~Okounkov and R.~Pandharipande. We also apply our methods to $b$-content Hurwitz numbers and for simple single Hurwitz numbers we retrieve a recent result of N.~Do, J.~He, and H.~Robertson, as well a proof of the coefficient gap in~\cite{DHR} in the case of possibly disconnected numbers.

\subsection{Acknowledgements}
D.~A. and D.~L. are supported by the University of Trieste, by the INdAM group GNSAGA, and by the INFN within the project MMNLP (APINE).
D.~A. and D.~L. thank Max Karev and Xiang Li for useful discussions.
The work of G.~R.~is funded by FCT - Fundação para a Ciência e a Tecnologia, I.P., through national funds, under the project UID/04561/2025. G.~R.~is grateful to Massimo Gisonni and Tamara Grava for valuable discussions.

\section{Large genus of hypergeometric Hurwitz numbers}
\label{sec:hypergeometric}

The key to the large-genus asymptotics of hypergeometric Hurwitz numbers is contained in the following well-known lemma, relating the analytic properties of a meromorphic function to the asymptotics of its Taylor coefficients. It can be found in many textbooks of analytic combinatorics. Here, we refer to H.~Wilf's \textit{generatingfunctionology}~\cite{Wilf}, to which we direct the reader for the proof.

\begin{theorem}[Theorem~5.2.1 in~\cite{Wilf}]
\label{thm:Wilf}
    Let $f=f(z)$ be meromorphic in an open neighborhood $\mathfrak R$ of the origin $z=0$, with finitely many poles in $\mathfrak R$.
    Let $R > 0$ be the modulus of the pole(s)
of smallest modulus, and let $z_0, . . . , z_S$ be all of the poles of $f (z)$ whose
modulus is $R$. Further, let $R_1 > R$ be the modulus of the pole(s) of $f$ of next-smallest modulus. Then, for any $\epsilon>0$, as $r\to+\infty$ we have
\begin{equation}
[z^r]f(z) = [z^r]\biggl\lbrace\sum_{i=1}^S PP(f(z);z_i)\biggr\rbrace +O\bigl((R_1^{-1}+\epsilon)^r\bigr).
\end{equation}
Here, $PP(f(z);z_i)$ denotes the principal part of $f(z)$ near the pole $z_i$.
\end{theorem}

In the interest of clarity, we immediately point out the special case which is relevant for the large-genus expansion of hypergeometric Hurwitz numbers. (Actually, only the cases $T=1$ or $T=2$ will be relevant in the following analyses.)

\begin{corollary}
\label{corollary:asymptotic}
Assume that the rational function $f(z)$ has the form 
\begin{equation}
    f(z)=h(z)+\sum_{t=1}^T\sum_{j=1}^K\frac{A_j^{(t)}}{(1 - z\rho_t)^j} ,
\end{equation}
where $K\in\Z_{\geq 1}$, $\rho_1,\dots,\rho_T\in\mathbb{C}\setminus\lbrace 0\rbrace$ with $|\rho_t|=\rho>0$ (for $1\leq t\leq T$), and $h(z)$ is a rational function with no poles in a disk strictly containing the disk $|\rho z|<1$.
Then, 
\begin{equation}
   [z^r]f(z)\sim\frac{r^{K-1}}{(K-1)!}\sum_{t=1}^TA_K^{(t)}\rho_t^r
\end{equation}
as $r\to +\infty$ through values for which $\sum_{t=1}^TA_K^{(t)}\rho_t^r\not=0$.
\end{corollary}
\begin{proof}
    Since $PP\bigl(f(z);\rho^{-1}\bigr)=\sum_{j=1}^K\frac{A_j}{(1 - z\rho)^j}=\sum_{j=1}^KA_j\sum_{r\geq 0}\binom{r+j-1}{r}(\rho z)^r$, we have
    \begin{equation}
    [z^r]PP\bigl(f(z);\rho^{-1}\bigr) = \rho^r\sum_{j=1}^KA_j\binom{r+j-1}{r}\sim A_K\frac{r^{K-1}}{(K-1)!}\rho^r
    \end{equation}
    where the last asymptotic relation holds as $r\to+\infty$.
    The statement follows immediately from Theorem~\ref{thm:Wilf}.
\end{proof}

\begin{remark}
\label{remark:refinederror}
    In the general Theorem~\ref{thm:Wilf}, when $f(z)$ is rational one can apply the theorem itself to $f(z)-\sum_{i=1}^SPP\bigl(f(z);z_i)$, thereby improving the error term to $O(r^{K-1}R_1^{-r})$, where $K$ is the maximum order of poles of modulus $R_1$. Note that this formulation provides finer asymptotics with respect to the above corollary. For the sake of simplicity we will not deal explicitly with the error terms (however, see Remark~\ref{remark:refinederror2}).
\end{remark}
\begin{lemma}\label{lem:A:pm:case}
Let $K\in\mathbb{Z}_{\geq 1}$ and consider the rational function of $z$
\begin{equation}
G(z) =
 \frac{(1+u_1 z)\cdots(1+u_L z)}
 {(1-z)^K(1-v_1 z)\cdots(1-v_M z)}
\end{equation}
and assume the following conditions hold:
\begin{equation}
     \quad v_j < 1 \quad \text{ for }\quad j = 1 \dots, M.
\end{equation}
Then, the function 
\begin{align*}
f(z) &\coloneqq \sum_{r\ge 0} H_r^{\bullet, G} (\mu^{(1)}, \dots, \mu^{(N)})\, z^r
\\
&=
\sum_{\lambda\vdash d}
\biggl(\frac{ \dim \lambda}{d!}\biggr)^2\,
\biggl(\prod_{i=1}^N \frac{\chi_\lambda(\mathcal{C}_{\mu^{(i)}})}{\dim\lambda}\biggr)
\prod_{(i,j)\in Y_\lambda}
G\bigl(z(j-i)\bigr)
\end{align*}
satisfies the conditions of the previous corollary with $T=2$,
\begin{equation}
    \rho_1=d-1,\quad\rho_2=1-d,\quad\rho=d-1,
\end{equation}
and
\begin{align*}
A_K^{(t)} &=
\frac{\bigl((-1)^{t+1}\bigr)^{Nd - \sum_j \ell(\mu^{(j)})}(d-1)^{(d-2)K}}
{d!^2(d-2)!^K}\times
\\
&
\times \left(
\prod_{i=1}^L
\sum_{a\ge 0}
\biggl[\begin{matrix} d \\ d-a \end{matrix}\biggr]
\left(\frac{u_i}{d-1}\right)^a
\right)
\!
\left(
\prod_{j=1}^{M}
\sum_{b\ge 0}
\biggl\{\begin{matrix} b+d-1 \\ d-1 \end{matrix}\biggr\}
\left(\frac{v_j}{d-1}\right)^b
\right).
\end{align*}
\end{lemma}
\begin{proof}
First, it is clear that $f(z)$ is a rational function of $z$.
By the condition on the $v_j$, the pole closest to the origin of $G(z)$ is $z=1$.
Then, the pole closest to the origin of $G(z(j-i))$ is $z=1/(j-i)$, hence it is sufficient to note the $j-i$ (where $(i,j)$ runs in $Y_\lambda$) attains its extreme values $\pm(d-1)$ for the partitions $\lambda=d$ (corresponding to $+$) and $\lambda=(1^d)$ (corresponding to $-$) only.
Therefore, there are exactly two poles of $f(z)$ closest to the origin, which have order $K$ and are located at $\pm 1/(d-1)$.
Hence, the conditions of Corollary~\ref{corollary:asymptotic} are satisfied as indicated in the statement.
In particular, to compute the leading polar coefficients $A_K^{(t)}=\lim_{z\to \pm(d-1)^{-1}}\bigl(1\mp z(d-1)\bigr)^Kf(z)$ (denoting for the rest of this proof $(-1)^{t+1}=\pm$), let us first give the relevant quantities in the symmetric group associated with the two partitions $\lambda=(d),(1^d)$.
These partitions correspond (respectively) to the trivial and sign representation of $\symmgroup d$, hence $\dim(d)=1=\dim (1^d)$ and
\[
\chi_{(d)}(\mathcal{C}_\mu) = 1, \qquad \qquad \chi_{(1^d)}(\mathcal{C}_\mu) = (-1)^{d - \ell(\mu)},
\]
for any $\mu\vdash d$, such that
\begin{equation}
\biggl(\frac{\dim\lambda}{d!}\biggr)^2\,\prod_{i=1}^{N} \frac{\chi_\lambda(\mathcal{C}_{\mu^{(i)}})}{\dim\lambda} = \frac{(\pm 1)^{Nd-\sum_{j=1}^N \ell(\mu^{(j)})}}{d!^2}, \;
\begin{cases}
+ \; & \lambda = (d),
\\
- \; & \lambda = (1^d).
\end{cases}
\end{equation}
We finally compute
\begin{align*}
A_K^{(t)} &= \frac{(\pm 1)^{Nd-\sum_{j=1}^N \ell(\mu^{(j)})}}{d!^2} \lim_{z \to \pm (d-1)^{-1}} \left(1 \mp z(d-1) \right)^K \prod_{j = 1}^{d-1} G\left( \pm jz \right)
\\
&=
\frac{(\pm 1)^{Nd - \sum_j \ell(\mu^{(j)})}
}{
d!^2
(1-\tfrac{1}{d-1})^K\cdots
(1-\tfrac{d-2}{d-1})^K
}
\frac{\prod_{i=1}^L\prod_{s=1}^{d-1}
\left(1+\frac{s u_i}{d-1}\right)}{\prod_{j=1}^{M}\prod_{s=1}^{d-1}
\left(1-\frac{s v_j}{d-1}\right)}
\\
&=
\frac{(\pm 1)^{Nd - \sum_j \ell(\mu^{(j)})}(d-1)^{K(d-2)}}
{d!^2(d-2)!^K}
\frac{
\prod_{i=1}^L \prod_{s=1}^{d-1}
\left(1+\frac{s u_i}{d-1}\right)
}{
\prod_{j=1}^{M} \prod_{s=1}^{d-1}
\left(1-\frac{s v_j}{d-1}\right)
}.
\end{align*}
Applying the formulae for the generating series of Stirling numbers in Appendix \ref{sec:Jucys} one concludes the proof. 
\end{proof}

As a consequence, we obtain the following result.

\begin{theorem}\label{thm:large:g:monotone:K}
Let $K\in\mathbb{Z}_{\geq 1}$ and consider the rational function of $z$
\begin{equation}
G(z) =
 \frac{(1+u_1 z)\cdots(1+u_L z)}
 {(1-z)^K(1-v_1 z)\cdots(1-v_M z)}.
\end{equation}
Then, the large genus asymptotic of the monotone Hurwitz numbers read
\begin{align*}
\Big[
&
\prod_{i=1}^L u_i^{a_i}
\prod_{j=1}^{M} v_j^{b_j}
\Big]
 H_r^{\bullet, G}(\mu^{(1)}, \dots, \mu^{(N)})
\sim
\\
&\sim
2 \frac{r^{K-1}}{(K-1)!}\frac{(d-1)^{(d-2)K-\sum a_i-\sum b_j+r}}{d!^2 (d-2)!^K}
\prod_{i=1}^L \biggl[\begin{matrix} d \\ d - a_i \end{matrix}\biggr]
\prod_{j=1}^{M} \biggl\{\begin{matrix} b_j + d - 1 \\ d-1 \end{matrix}\biggr\},
\end{align*}
as $g\to +\infty$ with 
$$
r = 2g - 2 - d(N-2) + \sum_{j=1}^N \ell(\mu^{(j)}).
$$
\end{theorem}

\subsection{The case of Grothendieck dessins d'enfants}
A natural related question would be to compute large genus asymptotics of the Grothendieck dessins d'enfants or of Hurwitz numbers with multiple blocks of strictly monotone ramifications (and no weakly monotone blocks).
This asymptotic problem does not make sense (at least in the present context, where the degree $d$ of the cover is fixed).
In fact the strictly monotone condition exhaust in finite ramifications for finite degree, in the same way elementary symmetric polynomials are zero whenever their homogeneous degree overtakes the number of variables they are evaluated on:
$$
e_r(x_1, \dots, x_d) = 0 \qquad \text{ whenever } r > d,
$$
hence $e_r(\mathcal{J}_1, \dots,\mathcal{J}_d) = 0$ whenever $r \geq d$.
Now notice that by Riemann--Hurwitz $r = r(g)$ is a linear function of the genus, that becomes asymptotically large. This does not instead happen for complete homogeneous polynomials, and therefore, for weakly monotone ramification blocks.

\subsection{Specialisation to single monotone Hurwitz numbers}

The monotone Hurwitz numbers studied by C.~Yang in~\cite{CYang} can be obtained by setting $G(z)=(1-z)^{-1}$.
Under this specialisation, Theorem~\ref{thm:large:g:monotone:K} (with $K=1$) recovers the following result of \textit{op. cit.}.

\begin{theorem}
When $G(z)=(1-z)^{-1}$, we have
\[
H_r^{\bullet, G}(\mu)\big{|}_{a_i=b_j=0}
\sim
\frac{2(d-1)^{d-2}}{d!^2(d-2)!}
(d-1)^{2g - 2 + \ell(\mu) + d},\quad\text{as } g\to\infty,
\]
for $G(z) = (1-z)^{-1}$.
\end{theorem}

\begin{remark}
\label{remark:refinederror2}
    The statement in \cite{CYang} differs from the one above because of different normalizations. Namely, it differs by a factor $\mu_1 \cdots \mu_{\ell(\mu)}$, which simply accounts (at the geometric level of Riemann surfaces counting) to considering the ramification points of ramification profile given by the partition $\mu \vdash d$ over the same point of the base to be labelled or not. Instead, it agrees with the convention in \cite{DHR}, see \cite[Remark 3.9]{CYang}. The statement also differs by a factor $d!$, which accounts for re-labelling of the covering sheets.
    
    Moreover, The argument based on the general Theorem~\ref{thm:Wilf} also provides the error term explicitly in the case $K=1$, cf.~Remark~\ref{remark:refinederror}. Indeed, the poles of next-smallest modulus only come from the partitions $(d-1,1)$ and $(2,1^{d-1})$ and we also recover the error term given in~\cite{CYang}.
\end{remark}

\subsection{Harish-Chandra--Itzykson--Zuber matrix model correlators} \label{sec:HCIZ}
In \cite{GGPN} the Harish-Chandra--Itzykson--Zuber random matrix model is considered
\begin{align*}
    \mathcal{I}_N(y) := \int_{U(N)} e^{yN \Tr (A_N U B_N U^{-1}) } dU,
\end{align*}

as well as the free energy of the model

\begin{align*}
    F_N(y) := \frac{1}{N^2} \log \mathcal{I}_N(y)
\end{align*}
for $y$ inside $\Omega_N$ a simply connected open set in $\mathbb{C}$ which contains the origin and avoiding the zeros of $\mathcal{I}_N(y)$. The logarithm is taken as the principal branch in $\Omega_N$, $U(N)$ is the group of $N \times N$ unitary matrices against the normalized Haar measure and $A_N, B_N$ are any two $N \times N$ complex diagonal matrices.

One of the main results of \cite{GGPN} states that the coefficients of the large $N$ expansion of $\mathcal{I}_N(y)$ are possibly disconnected monotone double Hurwitz numbers, in the same way the coefficients of the large $N$ expansion of $F_N(y)$ are connected monotone double Hurwitz numbers. More precisely it is proved that there is a large $N$ expansion of the form
\begin{equation}
    \mathcal{I}_N(z) = 1 + \sum_{d \ge 1}^N \frac{y^d}{d!} \sum_{r \ge 0} \left( - \frac{1}{N}\right)^r \sum_{\mu, \nu \vdash d} p_{\mu}(A_N) p_{\nu}(B_N) W_d^r(\mu, \nu) + O(z^{N+1})
\end{equation}
where $p_{\mu}(A_N) := \prod_{i=1}^{\ell(\mu)} \Tr(A_N^{\mu_i})$ and $(A_N, B_N)_{N \ge 1}$ are two sequences of complex diagonal matrices which grow in a suitably regular way (we refer to the original paper for more details) and most importantly that (for $G(z)=(1-z)^{-1}$)
\begin{equation}
    W_d^r(\mu, \nu) = H_r^{G, \bullet}(\mu, \nu) = [z^r]. \sum_{\lambda \vdash d} 
    \frac{\chi_\lambda(\mathcal{C}_\mu)\chi_\lambda(\mathcal{C}_\nu)}{(d!)^2} 
\prod_{(i,j)\in Y_\lambda} \frac{1}{1 - z(j-i)}
\end{equation}
For fixed $d$ and large $r = r(g)$ or equivalently large genus $g$, section \ref{sec:hypergeometric} hence provides the large $N$ asymptotics of the Harish-Chandra--Itzykson--Zuber matrix model and its free energy.

\subsection{ \texorpdfstring{$b$}{b}-content Hurwitz numbers}
\label{sec:b}
In this section we define the $b$-content Hurwitz numbers of hypergeometric type and compute their large genus asymptotics, adapting our methods from previous sections. 
$b$-Hurwitz numbers are a deformation of Hurwitz numbers, the latter being retrieved when the parameter $b$ is set to 1.
It is worth remarking that the definition we give (following~\cite{CD}) is, in a certain sense, \emph{formal}.
Namely, it is obtained by the following $b$-deformation of the ingredients entering~\eqref{eq:Frobenius}:
\begin{align*}
\biggl(\frac{d!}{\dim\lambda}\biggr)^2 &\mapsto j^{(b+1)}_\lambda,
\\
\frac{\chi_\lambda(\mathcal{C}_\mu)}{\dim\lambda} &\mapsto \widetilde\theta^{(b+1)}_\lambda(\mu),
\\
j-i&\mapsto (b+1)(j-1)-(i-1).
\end{align*}
To explain these notations, recall that the Jack symmetric functions $J^{(\alpha)}_\lambda$ are a famous deformation of Schur functions, introduced in the seminal work~\cite{Jack} (see also~\cite{Macdonald,StanleyJack}) which specialise to the Schur symmetric functions when $\alpha=1$ and to the Zonal functions when $\alpha=2$.
We refer the reader to~\cite{GouldenJacksonb,DolegaFeray,BenDaliIntegrality} for some background motivation for this deformation.

Then, $j^{(\alpha)}_\lambda$ is the squared norm in the deformed Hall product of the Jack functions (with respect to which, the latter are orthogonal), explicitly given by
\begin{equation}
\label{eq:jJack}
    j^{(\alpha)}_\lambda = \biggl(\prod_{(i,j)\in Y_\lambda}\bigl(\alpha(\lambda_i-j)+\lambda'_j-i+1\bigr)\biggr)
    \biggl(\prod_{(i,j)\in Y_\lambda}\bigl(\alpha(\lambda_i-j)+\lambda'_j-i+\alpha\bigr)\biggr).
\end{equation}
where we recall that $Y_\lambda$ is the diagram of $\lambda$, defined in~\eqref{eq:Ylambda}.
That $j^{(1)}_\lambda=\bigl(d!/\dim\lambda\bigr)^2 $ when $\alpha=1$ is seen, for example, by the hook-product formula for $\dim\lambda$.

Moreover, $\widetilde\theta^{(\alpha)}_\lambda(\mu)$ are (a rescaling of) the so-called \emph{Jack characters}~\cite{DolegaFeray}, defined by the expansion 
\begin{equation}
J^{(\alpha)}_\lambda\,=\,\sum_{\mu\vdash d} |\mathfrak{C}_\mu| \; \tilde{\theta}^{(\alpha)}_\lambda(\mu)\,p_\mu,
\end{equation}
of Jack symmetric functions of parameter $\alpha$ onto the power sum basis $\lbrace p_\mu\rbrace_{\mu\in\partitions}$ of the ring of symmetric functions. By the Frobenius formula, the analogous expansion of Schur functions is given by the characters of the symmetric group, and so $\widetilde{\theta}^{(1)}_{\lambda}=\chi_\lambda(\mathcal{C}_\mu)/\dim\lambda$. For more details we refer to the literature we indicated above, e.g.,~\cite{DolegaFeray}.

\begin{remark}
The representation theoretic interpretation of the $b$-Hurwitz numbers would require a theory of $b$-Jucys--Murphy elements, which, to the best of our knowledge, have not been completely defined for general $b$.
However there are general elements called $(C,J,T)$-Jucys--Murphys elements \cite{FHKM} which are conjecturally supposed to define $b$-Jucys--Murphy elements for specific values of the variables $C,J,T$ (namely, or $CJ = \frac{b+1}{2}$ and $T = b$). Since the development of a $b$-theory for Jucys--Murphy elements is beyond the purpose of this paper, we apply instead our methods to compute the large genus asymptotics of a formal object while waiting for their full connection with geometry and representation theory to be established.
Note however that some combinatorial and geometric interpretation for the $b$-Hurwitz numbers has already appeared, see~\cite{CD,BCD1,BCD2,BenDaliGenerating,R}. In particular, for $b=1$, the following results should involve Zonal polynomials.
\end{remark}

\begin{definition}
Given $G\in 1+z\mathbb{Q}[[z]]$, we define the possibly disconnected hypergeometric $b$-content Hurwitz numbers through their generating series 
\begin{align*}
    f_b(z) ={}& \sum_{r \geq 0} H^{\bullet, G,b}_r(\mu^{(1)}, \dots, \mu^{(N)}) z^r
\\
:={}&
        \sum_{\lambda\vdash d}
        \frac{ 1}{j^{(b+1)}_\lambda}\,
        \biggl(\prod_{i=1}^N \widetilde\theta^{(b+1)}_\lambda(\mu^{(i)})\biggr)
        \prod_{(i,j)\in Y_\lambda}
        G\left(z((b+1)(j-1)-(i-1))\right).
    \end{align*}
\end{definition}

To apply the general idea of looking for the poles of smallest modulus of $f_b(z)$, let us first give a simple lemma.

\begin{lemma}
\label{lemma:simple}
    Let $\alpha\in\mathbb{C}$ and $d\in\Z_{\geq 1}$ and let $$R_\alpha(\lambda)=\max_{(i,j)\in Y_\lambda}|\alpha(j-1)-(i-1)|.$$

    Then we have the following cases:
    
    if $|\alpha|>1$, $R_\alpha(\lambda)< R_\alpha\bigl((d)\bigr)=|\alpha|(d-1)$ for all $\lambda\vdash d$ with $\lambda\not=(d)$;
    
    if $|\alpha|<1$, $R_\alpha(\lambda)< R_\alpha\bigl((1^d)\bigr)=d-1$ for all $\lambda\vdash d$ with $\lambda\not=(1^d)$;
    
    if $|\alpha|=1$, $R_\alpha(\lambda)<  R_\alpha\bigl((d)\bigr)=R_\alpha\bigl((1^d)\bigr)=d-1$ for all $\lambda\vdash d$ with $\lambda\not=(d),(1^d)$.
\end{lemma}
\begin{proof}
By the definition of $Y_\lambda$ in~\eqref{eq:Ylambda}, for any $(i,j)\in Y_\lambda$ we have
\begin{equation}
    |\alpha (j-1)-(i-1)|\leq |\alpha|(\lambda_i-1)+(i-1)\leq 
    \max\lbrace|\alpha|,1\rbrace (\lambda_i+i-2).
\end{equation}
For any $\lambda\vdash d$ we have $\lambda_i+i\leq d+1$ for all $i\leq \ell(\lambda)$, and the equality is possible for the partitions $\lambda=(d)$ or $\lambda=(1^d)$ only.
The statement then follows easily.
\end{proof}

Because of this lemma, again only the partitions $\lambda=(d)$ and $\lambda=(1^d)$ play a role in the asymptotic problem (to leading order), hence we need formulas for $j^{(\alpha)}_\lambda$ and $\widetilde\theta_{\lambda}^{(\alpha)}$ when $\lambda=(d)$ and $\lambda=(1^d)$.
For~$j^{(\alpha)}_\lambda$, directly from~\eqref{eq:jJack} we get
\begin{equation}
\label{eq:Jack1}
    j^{(\alpha)}_{(d)} \,=\,d!\,\alpha^{d}
    \prod_{m=1}^{d-1}(1+m\alpha),\qquad
    j^{(\alpha)}_{(1^d)} \,=\, \,d!
    \prod_{m=1}^{d-1}(1+m\alpha).
\end{equation}
For the Jack characters, the relevant special values are well-known:
\begin{equation}
    \label{eq:Jack2}
    \widetilde\theta_{(d)}^{(\alpha)}(\mu) \,=\, \alpha^{d-\ell(\mu)},\qquad
    \widetilde\theta_{(1^d)}^{(\alpha)}(\mu) \,=\, (-1)^{d-\ell(\mu)}.
\end{equation}
For the proof, see for example~\cite[Proposition~2.2(b) and Theorem 3.3]{StanleyJack}.

It is natural to assume $\Re\alpha>0$, for otherwise $j_\lambda^{(\alpha)}$ may vanish. This corresponds to assuming $\Re b>-1$; the next result naturally extends to a larger domain for $b$ but we do not enter into details here.
\begin{theorem}
Let $K\in\mathbb{Z}_{\geq 1}$ and consider the rational function of $z$
\begin{equation}
G(z) =
 \frac{(1+u_1 z)\cdots(1+u_L z)}
 {(1-z)^K(1-v_1 z)\cdots(1-v_M z)}.
\end{equation}
Then, for any $b\in\mathbb C$ with $\Re b>-1$, the large genus asymptotic of the associated hypergeometric $b$-content Hurwitz numbers is described by:
\begin{align*}
&\Big[
\prod_{i=1}^L u_i^{a_i}
\prod_{j=1}^{M} v_j^{b_j}
\Big] H^{\bullet, G, b}_r
\sim
\\
&\sim
\frac{r^{K-1}}{(K-1)!} 
\frac{(d-1)^{K(d-2)-\sum a_i-\sum b_j+r}}
{d!\bigl(\prod_{m=1}^{d-1}(1+(b+1)m)\bigr)(d-2)!^K}
\prod_{i=1}^L \biggl[\begin{matrix} d \\ d - a_i \end{matrix}\biggr]
\prod_{j=1}^{M} \biggl\{\begin{matrix} b_j + d - 1 \\ d-1 \end{matrix}\biggr\}\,\times
\\
&\quad\times\,
\begin{cases}
(b+1)^{r+(N-1)d - \sum_j \ell(\mu^{(j)})}, & |b+1|>1; 
\\ 
(-1)^{r+Nd - \sum_j \ell(\mu^{(j)})}, & |b+1|<1 ; 
\\
(b+1)^{r+(N-1)d - \sum_j \ell(\mu^{(j)})}+(-1)^{r+Nd - \sum_j \ell(\mu^{(j)})}, & |b+1|=1; 
\end{cases}
\end{align*}
as $r\to +\infty$.
\end{theorem}
\begin{proof}
The proof follows along similar lines as the proof of Theorem~\ref{thm:large:g:monotone:K}.
Let us consider separately the cases $|b+1|>1$, $|b+1|<1$, and $|b+1|=1$, and assume $|a_i|<1$ and $|b_j|<1$ (for $1\leq i\leq L$ and $1\leq j\leq M$).

First, if $|b+1|>1$, Lemma~\ref{lemma:simple} implies that $f_b(z)$ satisfies the hypothesis of Corollary~\ref{corollary:asymptotic} with $T=1$, $\rho_1=(b+1)(d-1)$, and $A_K^{(1)}$ given by (using~\eqref{eq:Jack1} and~\eqref{eq:Jack2})
\begin{align*}
A_K^{(1)} 
&= \lim_{z \to ((b+1)(d-1))^{-1}} \frac{(b+1)^{(N-1)d-\sum_{j=1}^N\ell(\mu^{(j)})}}{d!\prod_{m=1}^{d-1}\bigl(1+(b+1)m\bigr)} \Big( 1 - z(b+1)(d-1) \Big)^K \prod_{j = 1}^{d-1} G\left( jz(b+1) \right)
\\
&= \lim_{z \to (d-1)^{-1}}  \frac{(b+1)^{(N-1)d-\sum_{i=1}^N\ell(\mu^{(i)})}}{d!\prod_{m=1}^{d-1}\bigl(1+(b+1)m\bigr)} \Big( 1 - z(d-1) \Big)^K \prod_{j = 1}^{d-1} G\left( jz\right)
\\
&=
\frac{(b+1)^{(N-1)d - \sum_j \ell(\mu^{(j)})}
}{
d!\prod_{m=1}^{d-1}\bigl(1+(b+1)m\bigr)}
\frac{1}{
(1-\tfrac{1}{d-1})^K\cdots
(1-\tfrac{d-2}{d-1})^K
}
\frac{\prod_{i=1}^L\prod_{s=1}^{d-1}
\left(1+\frac{s u_i}{d-1}\right)}{\prod_{j=1}^{M}\prod_{s=1}^{d-1}
\left(1-\frac{s v_j}{d-1}\right)}
\\
&=
\frac{(b+1)^{(N-1)d - \sum_j \ell(\mu^{(j)})}(d-1)^{K(d-2)}}
{d!\biggl(\prod_{m=1}^{d-1}\bigl(1+(b+1)m\bigr)\biggr)(d-2)!^K}
\frac{
\prod_{i=1}^L \prod_{s=1}^{d-1}
\left(1+\frac{s u_i}{d-1}\right)
}{
\prod_{j=1}^{M} \prod_{s=1}^{d-1}
\left(1-\frac{s v_j}{d-1}\right)
}.
\end{align*}

Next, if $|b+1|<1$, Lemma~\ref{lemma:simple} implies that $f_b(z)$ satisfies the hypothesis of Corollary~\ref{corollary:asymptotic} with $T=1$, $\rho_1=1-d$, and $A_K^{(1)}$ given by (using~\eqref{eq:Jack1} and~\eqref{eq:Jack2})
\begin{align*}
A_K^{(1)} 
&= \lim_{z \to (d-1)^{-1}}  \frac{(-1)^{Nd-\sum_{i=1}^N\ell(\mu^{(i)})}}{d!\prod_{m=1}^{d-1}\bigl(1+(b+1)m\bigr)} \Big( 1 + z(d-1) \Big)^K \prod_{j = 1}^{d-1} G\left( -jz\right)
\\
&=
\frac{(-1)^{Nd - \sum_j \ell(\mu^{(j)})}
}{
d!\prod_{m=1}^{d-1}\bigl(1+(b+1)m\bigr)}
\frac{1}{
(1-\tfrac{1}{d-1})^K\cdots
(1-\tfrac{d-2}{d-1})^K
}
\frac{\prod_{i=1}^L\prod_{s=1}^{d-1}
\left(1+\frac{s u_i}{d-1}\right)}{\prod_{j=1}^{M}\prod_{s=1}^{d-1}
\left(1-\frac{s v_j}{d-1}\right)}
\\
&=
\frac{(-1)^{Nd - \sum_j \ell(\mu^{(j)})}(d-1)^{K(d-2)}}
{d!\biggl(\prod_{m=1}^{d-1}\bigl(1+(b+1)m\bigr)\biggr)(d-2)!^K}
\frac{
\prod_{i=1}^L \prod_{s=1}^{d-1}
\left(1+\frac{s u_i}{d-1}\right)
}{
\prod_{j=1}^{M} \prod_{s=1}^{d-1}
\left(1-\frac{s v_j}{d-1}\right).
}
\end{align*}
These computations show that in each case we can apply Corollary~\ref{corollary:asymptotic}, which then gives the statement of the theorem in a similar way as in the case $b=0$ discussed above. When $|b+1|=1$ we have to take into account the contribution of both poles and the statement follows in a similar manner.
\end{proof}

\section{Large genus of completed cycles Hurwitz numbers}
\label{sec:completed:cycles}

In this section we compute the large genus asymptotics of Hurwitz numbers with an arbitrary number $N$ of fixed ramifications and $(s+1)$-completed cycles as remaining ramifications.
Through the Gromov--Witten / Hurwitz correspondence we provide large genus asymptotics of the stationary sector of the Gromov--Witten theory of $\mathbb{CP}^1$, which can be extended to smooth algebraic curves of higher genus via the degeneration formula.

Consider the $(s+1)$-completed cycles Hurwitz numbers with 
\begin{equation}
\label{eq:CCHN}
H_r^{\bullet, s+1} \bigl(\mu^{(1)},\dots,\mu^{(N)}\bigr) :=
 \sum_{\lambda\vdash d} \biggl(\frac{\dim\lambda}{d!}\biggr)^2\biggl(\prod_{i=1}^N\frac{\chi_{\lambda}(\mathcal{C}_{\mu^{(i)}})}{\dim\lambda}\biggr)  \overline{f}_{s+1}(\lambda)^{r}
\end{equation}
where, in terms of the shifted Frobenius coordinates $(R;a'_1,\ldots,a_R';b'_1,\ldots,b_R')$ of $\lambda$ (see Section~\ref{sec:Shifted:Frob}) we have
\begin{equation}
\label{eq:defCCf}
    \overline{f}_{s}(\lambda)\,=\, \frac 1{s}\Biggl(\sum_{i=1}^{R} \left[ \bigl(a_i'\bigr)^{s} - \bigl(-b_i'\bigr)^{s} \right]\,+\,c_s\Biggr),\qquad c_s=(1-2^{-s})\zeta(-s)\,.
\end{equation}

The only dependence on the genus $g$ is in the exponent $r$:
by Riemann--Hurwitz,
\begin{equation}
\label{eq:RHcc}
    rs = 2g - 2 - d(N-2) + \sum_{i=1}^N \ell(\mu^{(i)})\,.
\end{equation}
Hence
\begin{equation}
r \sim \frac{2}{s} \cdot g,\qquad g\to+\infty.
\end{equation}

Similar to the previous section, we will see that only the trivial partitions contribute to the leading asymptotics.
Indeed, we have the following lemma.

\begin{lemma}
\label{lemma:CC}
    Fix $d\in\mathbb{Z}_{\geq 1}$.
    Then, for all $s>1$ and for all $\lambda\in\partitions_d$ with $\lambda\not=(d),(1^d),(d-1,1),(2,1^{d-2})$, we have
    \begin{equation}
\bigl|\overline f_{s}(\lambda)\bigr|<\bigl|\overline f_{s}\bigl((d-1,1)\bigr)\bigr|
=\bigl|\overline f_{s}\bigl((2,1^{d-2})\bigr)\bigr|<\bigl|\overline f_{s}\bigl((d)\bigr)\bigr|
=\bigl|\overline f_{s}\bigl((1^d)\bigr)\bigr|\,.
\end{equation}
\end{lemma}
\begin{proof}
First consider $s>1$ even, such that the function $\phi(x)=(x+1)^{s}-x^{s}$ is increasing for all $x\in\mathbb{R}$.
Using the equivalent formula
\begin{equation}
\overline{f}_{s}(\lambda)=\frac 1{s}\sum_{i=1}^\infty
\biggl[\biggl(\lambda_i-i+\frac 12\biggr)^{s}-\biggl(-i+\frac 12\biggr)^{s}\biggr]
\end{equation}
(which is actually a finite sum and where we note that $\zeta(-s)=0$ for even $s$) we see that if $\lambda,\widetilde{\lambda}\in\partitions_d$ are related by
    \begin{equation}
\widetilde{\lambda}_i = \lambda_i + 1,\qquad
\widetilde{\lambda}_j = \lambda_j - 1,
    \end{equation}
    for some $i<j$ (while $\widetilde{\lambda}_k=\lambda_k$ for all $k\not=i,j$), we have
    \begin{equation}
    \overline f_{s}(\widetilde{\lambda})-\overline f_{s}(\lambda)=\frac{\phi(\lambda_i)-\phi(\lambda_j-1)}{s}>0.
    \end{equation}
Starting from any partition $\lambda$, we can apply iteratively transformations of the type $\lambda\mapsto\widetilde{\lambda}$ to reach the partition $(d)$.
We conclude that the maximum is obtained by the partition $(d)$ and the only case next-to-maximum comes from $(d-1,1)$, which is the only partition from which the only admissible transformation $\lambda\mapsto\widetilde{\lambda}$ of the form above cannot reach any partition other than~$(d)$. The proof is complete in the case $s$ even by noting that $\overline f_{s}(\lambda^T)=-\overline f_{s}(\lambda)$ (where $\lambda^T$ is the transposed partition of $\lambda$).

On the other hand, when $s$ is odd we have
\begin{equation}
\overline{f}_{s}(\lambda) = \frac 1{s}\Biggl(\sum_{i=1}^{R} \left[ \bigl(a_i'\bigr)^{s} +\bigl(b_i'\bigr)^{s} \right]\,+\, c_s \Biggr)
\end{equation}
subject to the constraints $a_i',b_i'\in\mathbb{Z}+\frac 12$, $\sum_{i=1}^R(a_i'+b_i')=d$, $a_i'>a_{i+1}'\geq \frac 12$, $b_i'>b_{i+1}'\geq \frac 12$.
Since the inequalities
\begin{equation}
    (x+y\pm \tfrac 12)^{s}-x^{s}-y^{s}>\mp \tfrac 12
\end{equation}
hold for all $x,y\geq \tfrac 12$ and for all $s>1$, we see that the transformation from the partition $\lambda$ with shifted Frobenius coordinates $(R;a_1',\dots,a_R';b_1',\dots,b_R')$ to the partition $\widetilde{\lambda}$ with shifted Frobenius coordinates $(R-1;a_1'+a_R'+\tfrac 12,a_2',\dots,a_{R-1}';b_1'+b_R'-\tfrac 12,b_2',\dots,b_{R-1}')$ necessarily satisfies $\overline f_{s}(\widetilde{\lambda})>\overline{f}_{s}(\lambda)$.
Hence, we can always strictly increase the value of $\overline {f}_{s}$ until we reach a hook partition, namely, a partition of the form $(1;a',b')$ in shifted Frobenius coordinates.
Since the inequality
\begin{equation}
(x+1)^{s}+(y-1)^{s}-x^{s}-y^{s}>0    
\end{equation}
holds for all $x\geq y\geq 1$ and all $s>1$, if $a'\geq b'\geq 3/2$ we can strictly increase the value of $\overline f_{s}$ by going to the partition $(1;a'+1;b'-1)$ while if $b'\geq a'\geq 3/2$ we can strictly increase the value of $\overline f_{s}$ by going to the partition $(1;a'-1;b'+1)$. The statement now follows easily.
\end{proof}

This leads to the following theorem.

\begin{theorem}\label{thm:large:g:completed} We have the following large genus asymptotics for the completed cycles Hurwitz numbers:
\begin{align*}
    H_r^{\bullet, s+1} &\bigl(\mu^{(1)},\dots,\mu^{(N)}\bigr) \sim 
    \frac{2}{d!^2\,(s+1)} \left[\left( d-\frac{1}{2} \right)^{s+1} - \left( -\frac{1}{2} \right)^{s+1} \,+\,\biggl(1-\frac{1}{2^{s+1}}\biggr)\zeta(-s-1)\right] ^r
\end{align*}
    for $sr =2g - 2 - d(N-2) + \sum_{i=1}^N \ell(\mu^{(i)})$ as $g \to \infty$.
\begin{proof}
By~\eqref{eq:CCHN} we see that $H_r^{\bullet, s+1}\bigl(\mu^{(1)},\dots,\mu^{(N)}\bigr) $ is a linear combination of $\overline f_{s+1}(\lambda)^r$ (with coefficients independent of $r$).
Hence, the leading asymptotic as $r\to+\infty$ is given by the terms corresponding to maximal $|\overline{f}_{s+1}(\lambda)|$, which, as seen above, correspond to $\lambda=(d),(1^d)$.
We have
    \begin{equation}
    \label{eq:Mdsmax}
    \overline{f}_{s+1}\bigl((d)\bigr) = M_{d,s},\qquad
    \overline{f}_{s+1}\bigl((1^d)\bigr) = (- 1)^{s} M_{d,s},
    \end{equation}
where
\begin{equation}
\label{eq:Mds}
M_{d,s}:=
    \frac{1}{s+1} \left[\left( d-\frac{1}{2} \right)^{s+1} - \left( -\frac{1}{2} \right)^{s+1} \,+\,c_{s+1}\right].
\end{equation}
Summing these two contributions, using~\eqref{eq:Mdsmax}, we get
\begin{equation}
H_r^{\bullet, s+1}
\sim \frac 1{d!}\,\Bigl(1+(-1)^{rs+Nd-\sum_{i}\ell(\mu^{(i)})}\Bigr)\,\bigl(M_{d,s}\bigr)^r,\qquad r\to+\infty,
\end{equation}
and using the Riemann--Hurwitz formula~\eqref{eq:RHcc} (along with the value of $c_s$ in~\eqref{eq:defCCf}) we complete the proof.
\end{proof}
\end{theorem}

As a corollary for $s=1$ and $N=1$ we retrieve the following recent theorem by N.~Do, J.~He and H.~Robertson \cite[Theorem 1.6]{DHR}.

\begin{theorem}
\begin{equation}
    H_r^{\circ, 2} \bigl( \mu \bigr) \sim \frac{2}{d!^2} \binom{d}{2}^{2g - 2 + d + \ell(\mu)} 
    \quad \text{ as } g \to \infty.
\end{equation}
\begin{proof}
By substituting $s=1$ and $N=1$ in Theorem~\ref{thm:large:g:completed} we find
    \begin{align*}
        H_r^{\circ, 2} \bigl( \mu \bigr) \sim \frac{2}{d!^2} \left(\frac{1}{2} \left[\left( d-\frac{1}{2} \right)^{2} - \left( -\frac{1}{2} \right)^{2} \right] \right)^{2g - 2 + d + \ell(\mu)} 
        =\frac{2}{d!^2} \binom{d}{2}^{2g - 2 + d + \ell(\mu)} 
    \end{align*}
    as desired. Connected and possibly disconnected numbers have the same large genus asymptotics by inclusion-exclusion formula.
\end{proof}
\end{theorem}

\subsection{Disconnected structure theorem and the coefficient gap}

\begin{theorem}[Structure theorem for disconnected completed cycles Hurwitz numbers]
\label{thm:structure:disconnected}
For all integers $s\geq 1$, we have
\begin{equation}
         H_r^{s+1, \bullet}(\mu^{(1)}, \dots, \mu^{(N)}) = \frac{2}{d!^{2}}
         \sum_{\substack{{m=\overline f_{s}(\lambda)}\\{\lambda \vdash d}}} 
         C_s^{\bullet}(\vec{\mu}, m) \, m^{r},
    \end{equation}
    where $sr = 2g - 2 - d(N-2) + \sum_{i=1}^N \ell(\mu^{(i)})$
    and
    \begin{equation}
        C_s^{\bullet}\left(\vec{\mu}, m \right) 
        =
        \!\!\!\!\!\!\!\!\!\!
        \sum_{\substack{\lambda \vdash d \\ \overline{f}_{s+1}(\lambda) = \pm m}}
        \!\!\!\!\!\!\!\!\!\!
        \frac{(\dim\lambda)^{2-N}}{2}
        \prod_{i=1}^N\chi_{\lambda}(\mathcal{C}_{\mu^{(i)}}).
    \end{equation}
    In particular the leading coefficient is
    \[C_s^{\bullet}\left(\vec{\mu}, M_{d,s}\right) = 1\]
    and just after the leading coefficient we have the coefficient gap
    \begin{equation}\label{eq:aaa}
        C^{\bullet}_s(\vec{\mu}, m) = 0 \; \text{ for } \;  \overline{f}_{s+1}((d-1,1)) < m < \overline{f}_{s+1}((d)).
    \end{equation}
\end{theorem}
\begin{proof}
When $s$ is odd, we have $c_{s+1}=0$ and $\overline{f}_{s+1}(\lambda^T)=-\overline{f}_{s+1}(\lambda)$ and so, by definition~\eqref{eq:CCHN} grouping together the partitions $\lambda$ and $\lambda^T$ (self-transposed partitions do not contribute), and using Lemma~\ref{lemma:CC}, we have
\begin{align}
    \nonumber
    H_r^{s+1, \bullet}(\mu^{(1)}, \dots, \mu^{(N)}) 
    &=
        \frac 1{d!^2}
        \!\!\!\!\!\!
        \sum_{
        m_i \in \overline{f}_s(\mathcal{P}_{d_i})}
        \!\!\!\!\!
        m^r
        \!\!\!\!\!
        \sum_{\substack{\lambda \vdash d \\ \overline{f}_{s+1}(\lambda) = m}}
        \!\!\!\!\!\!
        (\dim\lambda)^{2-N}\bigl(1+(-1)^{r+Nd-\sum_{i=1}^N\ell(\mu^{(i)})}\bigr)\prod_{i=1}^N\chi_\lambda(\mathcal C_{\mu^{(i)}})
\\  
&=\frac 2{d!^2}\sum_{m_i \in \overline{f}_s(\mathcal{P}_{d_i})}m^r\sum_{\substack{\lambda \vdash d \\ \overline{f}_{s+1}(\lambda) = m}}(\dim\lambda)^{2-N}\prod_{i=1}^N\chi_\lambda(\mathcal C_{\mu^{(i)}})
\end{align}
where we used Riemann--Hurwitz formula~\eqref{eq:RHcc} and that 
\[\chi_{\lambda^T}(\mathcal C_{\mu})\,=\,\chi_{(1^d)}(\mathcal C_\mu)\,\chi_\lambda(\mathcal C_\mu)\,=\,(-1)^{d-\ell(\mu)}\,\chi_\lambda(\mathcal C_\mu)\]
stemming from the well-known fact that the irreducible representations labeled by $\lambda$ and $\lambda^T$ are related by tensor product with the sign representation (labeled by $(1)^d$) --- a fact which implies the identity $\det\lambda=\det\lambda^T$, which we also used.

When $s$ is even, we have $\overline{f}_{s+1}(\lambda^T)=\overline{f}_{s+1}(\lambda)$, and we compute, in a similar way (without grouping together the partitions $\lambda,\lambda^T$ because in this case self-transposed partitions contribute),
\begin{align}
\nonumber
H_r^{s+1, \bullet}(\mu^{(1)}, \dots, \mu^{(N)}) =\frac 1{d!^2}\sum_{\substack{ \overline{f}_{s+1}(\lambda) = m \\ \lambda \vdash d}}m^r\sum_{\substack{\lambda \vdash d \\ \overline{f}_{s+1}(\lambda) = m}}(\dim\lambda)^{2-N}\prod_{i=1}^N\chi_\lambda(\mathcal C_{\mu^{(i)}})
\end{align}
and the proof is complete.
\end{proof}


We now employ the inclusion-exclusion principle to transfer the result above to the connected numbers.

\begin{proposition}[Polynomial connected Structure statement] We have 
\label{thm:structure:connected}
     \begin{equation}\label{eq:connected:hur}
        \widetilde{H}^{s+1,\circ}_r(\vec\mu) 
        \,=\,
        \frac{2}{d!^2} 
        \!\!\!\!\!\!\!\!\!\!\!\!
        \sum_{\substack{
        k \geq 1
        \\
        m_1,\dots, m_k \in \overline{f}_s(\mathcal{P}_d)
        \\
        {r_1,\dots, r_k \geq 1}}}
        \!\!\!\!\!\!\!\!\!\!\!\!
        C^{\circ}_{s, \vec{r}}(\vec{\mu}, \vec{m})
        \, \, \prod_{i=1}^k m_i^{r_i}
    \end{equation}
    where 
    \begin{equation} \label{eq:connected:coeff}
        C^{\circ}_{s, \vec{r}}(\vec{\mu}, \vec{m})
        = \frac{(-2)^{\ell(\vec{r})-1}}{\ell(\vec{r})}
        \!\!\!\!\!\!\!\!\!\!\!\!\!\!
        \!\!\!\!\!\!\!\!\!\!\!\!\!\!\!\!\!
        \sum_{\substack{
        \phantom{h}
        \\
        d_1, \dots, d_{\ell(\vec{r})} \geq 1 : \sum_i d_i = |\mu^{(1)}|
        \\
        \vec{\sigma_1},\dots, \vec{\sigma}_{\ell(\vec{r})} : \sigma^{(j)}_1\sqcup\cdots\sqcup\sigma^{(j)}_{\ell(\vec{r})}=\mu^{(j)} \, \wedge \; \sigma^{(j)}_i\vdash d_i
        \\
        g_1 + \cdots + g_{\ell(\vec{r})} = g + \ell(\vec{r}) - 1
        \\
        r_i s = 2g_i - 2 - d_i( \ell(\vec{\mu}) - 2 ) +\sum_{j=1}^{\ell(\vec{\mu})} l(\sigma_i^{(j)})
        }}
        \!\!\!\!\!\!\!\!\!\!\!\!\!\!\!\!\!\!\!\!\!\!\!\!\!\!\!\!\!\!
        \binom{d}{d_1\,\cdots\,d_{\ell(\vec{r})}}^2
        \binom{r}{r_1\,\cdots\,r_{\ell(\vec{r})}}
        \prod_{i=1}^{\ell(\vec{r})} C^\bullet_s(\vec\sigma_i,m_i).
    \end{equation}

    \begin{remark}
        The sums in \ref{eq:connected:coeff} and \ref{eq:connected:hur} are finite. 
        In fact $k$ is bounded by $d$, since any of the partitions $\mu^{(i)}$ is of size $d$, and $k$ is the number of the $\sigma$ disjoint subpartitions.
        Each $m_i$ varies over the values of $\bar{f}_s(\lambda)$. 
        And the $(r_i)_i$ are positive and bounded above by the Riemann--Hurwitz formula $r_is \leq 2g+k+dN$.
    \end{remark}
    
\begin{proof}
    We are going to apply the inclusion-exclusion principle to Theorem \ref{thm:structure:disconnected} to transfer the structure theorem from disconnected to the connected Hurwitz numbers. Introduce the set of variables $\mathbf{p} = \{p_{i}^{(j)}\}_{i \geq 0,j=1,\dots,N}$ and with $p_{0}^{(j)}=1$.
    For a vector of partitions $\vec \mu=\bigl(\mu^{(1)},\dots,\mu^{(N)}\bigr)\in(\partitions_d)^N$, we set
    \[p_{\vec\mu}\,=\,\prod_{j=1}^N \bigl(p_{\mu_1^{(j)}}p_{\mu_2^{(j)}}p_{\mu_3^{(j)}}\cdots\bigr)\]
    and introduce the possibly disconnected generating function
    \begin{equation}
        \mathbf{H}^{s+1,\bullet}(x,\mathbf p) = 1+\sum_{r\geq 0}\sum_{d\geq 1}\sum_{\vec{\mu}} \widetilde H^{s+1,\bullet}_r(\vec\mu)  \frac{x^r}{r!} p_{\vec\mu}
    \end{equation}
    as well as the connected generating function
    \begin{equation}
        \mathbf{H}^{s+1,\circ}(x,\mathbf p) = \sum_{r\geq 0}\sum_{d\geq 1}\sum_{\vec{\mu}} \widetilde H^{s+1,\circ}_r(\vec\mu)  \frac{x^r}{r!} p_{\vec\mu}
    \end{equation}
    where both for connected and possibly disconnected we have set
    \begin{equation}
    \widetilde{H}_r^{s+1}(\vec\mu) \, \coloneqq \, \biggl(\prod_{j=1}^N |\mathfrak{C}_{\mu^{(j)}}|\biggr)\,H^{s+1}_r(\vec\mu)\,. 
    \end{equation}
    Then it is well-known that the inclusion-exclusion relation between connected and disconnected Hurwitz numbers can be encoded in the algebraic relation
    \begin{equation}
         \mathbf{H}^{s+1,\circ}(x,\mathbf p) \,=\,\log \mathbf{H}^{s+1,\bullet}(x,\mathbf p) \,,
    \end{equation}
    see, e.g.,~\cite[Chapter~10]{CM}.
    Expanding the logarithm on the right-hand side we obtain
    \begin{equation}
        \widetilde{H}^{s+1,\circ}_r(\vec\mu) 
        \,=\,
        \sum_{k\geq 1}\frac{(-1)^{k-1}}{k}
        \!\!\!\!\!\!\!\!\!\!\!\!\!\!\!\!\!\!\!\!\!\!\!\!\!\!
        \sum_{\substack{(r_1,\vec\sigma_1),\ldots,(r_k,\vec\sigma_k) 
        \\
        r_i s = 2g_i - 2 - d_i(N-2) + \sum_{j=1}^N \ell(\sigma_i^{(j)})
        \\ 
        g_1 + \cdots + g_k = g + k  -1
        \\ 
        \sigma^{(j)}_i\vdash d_i \ \forall j=1,...,N
        \\ \sigma^{(j)}_1\sqcup\cdots\sqcup\sigma^{(j)}_k=\mu^{(j)}}}
        \!\!\!\!\!\!\!\!\!\!\!\!\!\!\!\!\!\!\!
        \binom{r}{r_1\,\cdots\,r_k}
        \prod_{i=1}^k\widetilde{H}_{r_i}^{s+1,\bullet}(\vec\sigma_i).
    \end{equation}
    such that
    \begin{equation}
        \widetilde{H}^{s+1,\circ}_r(\vec\mu) 
        \,=\,
        \sum_{k\geq 1}\frac{(-1)^{k-1}}{k}
        \!\!\!\!\!\!\!\!\!\!\!\!\!\!\!\!\!\!\!\!\!\!\!\!\!\!
        \sum_{\substack{(r_1,\vec\sigma_1),\ldots,(r_k,\vec\sigma_k)
        \\
        r_i s = 2g_i - 2 - d_i(N-2) + \sum_{j=1}^N \ell(\sigma_i^{(j)})
        \\ 
        g_1 + \cdots + g_k = g + k -1
        \\
        \sigma^{(j)}_i\vdash d_i \ \forall j=1,...,N
        \\ 
        \sigma^{(j)}_1\sqcup\cdots\sqcup\sigma^{(j)}_k=\mu^{(j)}
        \\
        m_i \in \overline{f}_s(\mathcal{P}_{d_i})
        }}
        \!\!\!\!\!\!\!\!\!\!\!\!\!\!\!\!\!\!\!\!\!
        \binom{r}{r_1\,\cdots\,r_k}
        \prod_{i=1}^k \frac{2}{(d_i!)^2}C^\bullet_s(\vec\sigma_i,m_i)m_i^{r_i}
    \end{equation}
    Rearranging the terms yields
    \begin{equation}
        \widetilde{H}^{s+1,\circ}_r(\vec\mu) 
        \,=\,
        \frac{2}{d!^2} \sum_{k\geq 1}\frac{(-2)^{k-1}}{k}
        \!\!\!\!\!\!\!\!\!\!\!\!\!\!\!\!\!\!\!\!\!\!\!\!\!\!
        \sum_{\substack{(r_1,\vec\sigma_1),\ldots,(r_k,\vec\sigma_k)
        \\
        r_i s = 2g_i - 2 - d_i(N-2) + \sum_{j=1}^N \ell(\sigma_i^{(j)})
        \\ 
        g_1 + \cdots + g_k = g + k-1 
        \\
        \sigma^{(j)}_i\vdash d_i \ \forall j=1,...,N
        \\ \sigma^{(j)}_1\sqcup\cdots\sqcup\sigma^{(j)}_k=\mu^{(j)}
        \\
        m_i \in \overline{f}_s(\mathcal{P}_{d_i})
        }}
        \!\!\!\!\!\!\!\!\!\!\!\!\!\!\!\!\!\!\!\!\!
        \binom{r}{r_1\,\cdots\,r_k}
        \binom{d}{d_1\,\cdots\,d_k}^2
        \prod_{i=1}^k C^\bullet_s(\vec\sigma_i,m_i)
        \prod_{i=1}^k m_i^{r_i}
    \end{equation}


Substituting the definition of the coefficients $C^{\circ}_{s, \vec{r}}$ concludes the proof of the proposition.
\end{proof}
\end{proposition}



\subsection{On the large genus of the orbifold Hurwitz numbers}
\label{sec:orbi}

By specialising the structure theorem above we obtain a new proof of \cite[Theorems 4.4 and 4.5]{DHR}. More precisely, let 
\begin{equation}
     s=1, \qquad N=2, \qquad \mu := \mu^{(1)}, \qquad \nu := \mu^{(2)} := \underbrace{(t, t, \dots, t)}_{d/t \text{ times}}, \quad t \in \mathbb{N}_{\geq 1}
\end{equation}
and indicate with $H_r^{\circ, [t]} \bigl( \mu \bigr)$ the resulting numbers. By convention whenever $d$ is not divisible by $t$, the corresponding Hurwitz number is assumed to be zero. Then we have:
\begin{equation}
          H_r^{\bullet, [t]} \bigl( \mu \bigr) = \frac{2}{d!^2}
         \sum_{1 \leq m \leq \binom{d}{2}} 
         C_1^{\bullet}((\mu, (t^{d/t})), m) \cdot m^{2g - 2 + \frac{d}{t} + \ell(\mu)}
\end{equation}

In particular for the large genus asymptotics we have:

\begin{equation}
    H_r^{\circ, [t]} \bigl( \mu \bigr) \sim H_r^{\bullet, [t]} \bigl( \mu \bigr) \sim \frac{2}{d!^2} \binom{d}{2}^{2g - 2 + \frac{d}{t} + \ell(\mu)} 
    \quad \text{ as } g \to \infty.
\end{equation}
Again, the only differences with the normalisation chosen in \cite{DHR} is the division by the product of the parts of the partitions $\mu$ and $\nu$, which accounts for a prefactor of 
\begin{equation}
    \frac{1}{(\prod_i \mu_i \prod_j \nu_j)} = \frac{1}{ t^{d/t} \cdot \prod_i \mu_i},
\end{equation}
and the division by the prefactor $(d/t)!$ instead of $d!$, which arises by the convention in \cite{DHR} of not labelling the preimages of the partitions $\nu$.

\subsection{On the large genus of the Gromov--Witten theory of \texorpdfstring{$\mathbb{CP}^{1}$}{the Riemann sphere}}
\label{sec:GW}
A.~Okounkov and R.~Pandharipande have solved the Gromov--Witten theory of smooth complex algebraic curves in the trilogy \cite{OP1,OP2,OP3}.
One of the key ingredients is the Gromov--Witten / Hurwitz correspondence, that over $\mathbb{CP}^1$ can be expressed as follows.

Define the stationary Gromov--Witten correlators of genus $g$ over the Riemann sphere, relative to two partitions $\mu$ and $\nu$ over $0$ and $\infty$ as
\begin{equation}
    \Big\langle \mu \big | \tau_{k_1} \cdots \tau_{k_n} \big | \nu \Big\rangle^{\circ, \mathbb{CP}^1}_{g} 
    :=
    \int_{[\overline{\mathcal{M}}_{g,n}(\mathbb{CP}^1, \mu, \nu)]^{\mathrm{vir}}}
    \prod_{i=1}^n \psi_i^{k_i} \mathrm{ev}_i^*(\omega)
\end{equation}
where we refer to the original papers for the notation details. (A specialisation of) the Gromov--Witten / Hurwitz correspondence can be expressed as 
\begin{align}
\label{eq:discGWP1}
    \Big\langle \mu \big | \tau_{2}^{m_2} \tau_{3}^{m_3} \dots \big | \nu \Big\rangle^{\bullet, \mathbb{CP}^1}_g
    = 
    \frac{1}{\mathfrak{z}(\mu)\mathfrak{z}(\nu)}\sum_{\lambda\vdash d} \frac{\chi_{\lambda}(\mathcal{C}_{\mu})\,\chi_{\lambda}(\mathcal{C}_{\nu})}{d!^2} \prod_{s \ge 1} \biggl(\frac{\overline{f}_{s+1}(\lambda)}{s!}\biggr)^{m_s}
\end{align}
where $\mathfrak{z}(\mu) = \prod_{i\geq 1} \bigl(m_i(\mu)!\,i^{m_i(\mu)}\bigr)$ is the stabilizer of the partition $\mu$ and $m_i(\mu)=\#\lbrace j\geq 1:\, \mu_j=i\rbrace$ is the multiplicity of $i$ in $\mu$.
(see Eq.~(3.1) in~\cite{OP1}) under the dimension constraint
    \begin{equation}
    \label{eq:dimconstraint}
    \sum_{s=2}^\infty \, s \, m_s \,=\, 2g - 2 + \ell(\mu) + \ell(\nu).
    \end{equation}

Hence, Lemma~\ref{lemma:CC} can be applied to this situation too to obtain the following theorem.

\begin{theorem} The stationary Gromov--Witten correlators of $\mathbb{CP}^1$ (relative to two partitions) are asymptotically 
    \begin{equation}
        \Big\langle \mu \big | \tau_{2}^{m_2} \tau_{3}^{m_3} \dots \big | \nu \Big\rangle^{\circ, \mathbb{CP}^1}_g 
        \sim
        \frac{2}{\mathfrak{z}(\mu)\,\mathfrak{z}(\nu)\,d!^2} 
        \prod_{s \ge 1}
        \left( \frac{M_{d,s}}{s!} \right)^{m_s}
    \end{equation}
    as $g\to+\infty$ in the regime where $m_k=0$ for $k > K$ for any fixed $K$ independent of $g$.
    We recall that, $M_{d,s}$ is defined in~\eqref{eq:Mds}.
\end{theorem}
\begin{proof}
    The proof is similar to the one given above to derive the large genus asymptotics of Hurwitz numbers with completed cycles. Indeed, due to the dimensional constraint~\eqref{eq:dimconstraint}, in the asymptotic regime in the statement at least one of the $m_i$ is going to $+\infty$.
    Then, Lemma~\ref{lemma:CC} implies that in the sum over partitions in the right-hand side of~\eqref{eq:discGWP1} the terms with $\lambda\not=(1^d),(d)$ are exponentially smaller than the terms with $\lambda=(1^d),(d)$ and so
    \begin{equation}
    \begin{aligned}
    &
    \Big\langle \mu \big | \tau_{2}^{m_2} \tau_{3}^{m_3} \dots \big | \nu \Big\rangle^{\bullet, \mathbb{CP}^1}_g 
    \\
    &\qquad\sim 
    \frac{1}{\mathfrak{z}(\mu)\,\mathfrak{z}(\nu)\,d!^2} \biggl[\prod_{s \ge 1} \Bigl(\frac{\overline{f}_{s+1}\bigl((d)\bigr)}{s!}\Bigr)^{m_s} + (-1)^{\ell(\mu)+\ell(\nu)}\prod_{s \ge 1} \Bigl(\frac{\overline{f}_{s+1}\bigl((1^d)\bigr)}{s!}\Bigr)^{m_s}\biggr]
    \\
    &\qquad=\frac{1+(-1)^{\ell(\mu)+\ell(\nu)+\sum_{s=1}^\infty sm_s}}{\mathfrak{z}(\mu)\,\mathfrak{z}(\nu)\,d!^2} \prod_{s \ge 1} \Bigl(\frac{\overline{f}_{s+1}\bigl((d)\bigr)}{s!}\Bigr)^{m_s}=\frac 2{d!^2} \prod_{s \ge 1} \Bigl(\frac{\overline{f}_{s+1}\bigl((d)}{s!}\bigr)\Bigr)^{m_s},
    \end{aligned}
    \end{equation}
    where we used the identity $\overline f_{s+1}\bigl((1^d)\bigr)=(-1)^s\overline f_{s+1}\bigl((d)\bigr)$ and~\eqref{eq:dimconstraint}, which shows that
    \begin{equation}
        \Big\langle \mu \big | \tau_{2}^{m_2} \tau_{3}^{m_3} \dots \big | \nu \Big\rangle^{\bullet, \mathbb{CP}^1}_g 
        \sim
        \frac{2}{\mathfrak{z}(\mu)\,\mathfrak{z}(\nu)\,d!^2} 
        \prod_{s \ge 1}
        \left(\frac{M_{d,s}}{s!}\right)^{m_s}.
    \end{equation}
    The statement then follows by observing that the large genus asymptotics of connected and possibly disconnected numbers coincide by the inclusion-exclusion formula.
\end{proof}

We conclude the paper with a couple of remarks on possible generalisations.

\begin{remark}
    Exploiting the degeneration formula in \cite{OP3}, one can extend the formula above to establish the large genus asymptotics of the stationary sector of the Gromov--Witten theory of an arbitrary smooth algebraic curve. This is achieved by reducing an arbitrary correlator of the stationary sector with $n$ descendant insertions to a genus $g$ component with no Gromov--Witten insertions with attached $n$ Riemann spheres, each with a single Gromov--Witten insertion and one relative condition \textemdash\; with kissing condition to the main genus $g$ component \textemdash\; and eventually multiplying  the corresponding contributions.
\end{remark}

\begin{remark}
    As apparent from the Burnside formula and as observed in \cite{XiangLi}, Hurwitz numbers enumerating coverings of a higher genus target $h$ can be obtained from the genus zero target ones by multiplying by the prefactor of $d!^{2h}$ and modifying the Riemann--Hurwitz count according to the transformation $r \to r - 2dh$.
\end{remark}

\appendix
\section{Jucys--Murphy elements, symmetric polynomials, Stirling numbers, and shifted Frobenius coordinates}
\label{sec:Jucys}

\subsection{Jucys--Murphy elements}

Let $\mathfrak{S}_d$ be the symmetric group on $d$ letters and $\mathbb{Q}[\mathfrak{S}_d]$ its group
algebra.
The \emph{Jucys--Murphy elements} are defined as
\[
\mathcal{J}_1 := 0, \qquad 
\mathcal{J}_k := \sum_{i=1}^{k-1} (i\,k), \qquad k=2,\dots,d,
\]
where $(i\,k)$ denotes the transposition exchanging $i$ and $k$. The elements $\mathcal{J}_1,\dots,\mathcal{J}_d$ pairwise commute and generate a maximal commutative
subalgebra of $\mathbb{Q}[\mathfrak{S}_d]$ and any symmetric polynomial evaluated at the Jucys--Murphy polynomials belong to the center of the group algebra. In any irreducible representation $V^\lambda$ of $\mathfrak{S}_d$, they act diagonally on the idempotent basis with eigenvalues given by the evaluation at the \emph{content}
$j-i$ of the boxes $(i,j)$ of the Young diagram $Y_\lambda$. See Jucys's theorem \cite{Jucys}.

Concretely, this implies that for any symmetric polynomial $F$ in $d$ variables, one has the following action of the center of the group algebra $\mathcal{Z}\mathbb{Q}[\mathfrak{S}_d]$ on the idempotent basis $\{\mathcal{F}_\lambda\}_{\lambda \vdash d}$, see~\eqref{eq:idempotent}:
\begin{equation}\label{eq:Jucys:eigen}
F(\mathcal{J}_1, \mathcal{J}_2,\dots,\mathcal{J}_d). \mathcal{F}_{\lambda}
=
F\bigl(\{j-i\}_{(i,j)\in Y_\lambda}\bigr)\,\mathcal{F}_{\lambda}, 
\end{equation}

\subsection{Elementary and complete symmetric polynomials}

Let $x_1,x_2,\dots$ be commuting variables.
The \emph{elementary symmetric polynomials} $e_k$ and the
\emph{complete homogeneous symmetric polynomials} $h_k$ are defined by the generating
series
\[
E(z) := \sum_{k\ge 0} e_k z^k = \prod_{i\ge 1} (1+x_i z),
\qquad
H(z) := \sum_{k\ge 0} h_k z^k = \prod_{i\ge 1} \frac{1}{1-x_i z}.
\]

Evaluating these symmetric polynomials at the Jucys--Murphy elements
$x_i = \mathcal{J}_i$ produces distinguished central elements of $\mathbb{Q}[\mathfrak{S}_d]$,
whose expansions in the class algebra are governed by Stirling numbers.

\subsection{Stirling numbers of the first and second kind}

The \emph{Stirling numbers of the first kind}
\(
\biggl[\begin{matrix} n \\ k \end{matrix}\biggr]
\)
are defined by the identity
\[
x(x-1)(x-2)\cdots(x-n+1)
=
\sum_{k=0}^n (-1)^{n-k}
\biggl[\begin{matrix} n \\ k \end{matrix}\biggr] x^k,
\]
and count permutations of $n$ elements with exactly $k$ cycles. The \emph{Stirling numbers of the second kind}
\(
\biggl\{\begin{matrix} n \\ k \end{matrix}\biggr\}
\)
are defined by
\[
x^n
=
\sum_{k=0}^n
\biggl\{\begin{matrix} n \\ k \end{matrix}\biggr\}
x(x-1)\cdots(x-k+1),
\]
and count set partitions of an $n$-element set into $k$ nonempty blocks. Their exponential generating functions are
\[
\sum_{n\ge k}
\biggl[\begin{matrix} n \\ k \end{matrix}\biggr]
\frac{z^n}{n!}
=
\frac{1}{k!}\bigl(\log(1+z)\bigr)^k,
\qquad
\sum_{n\ge k}
\biggl\{\begin{matrix} n \\ k \end{matrix}\biggr\}
\frac{z^n}{n!}
=
\frac{1}{k!}\bigl(e^z-1\bigr)^k.
\]

\subsection{Jucys--Murphy elements and Stirling numbers}

A fundamental result of Jucys states that evaluating elementary and complete
homogeneous symmetric polynomials at the Jucys--Murphy elements yields Stirling
numbers in the class algebra:
\[
e_k(\mathcal{J}_1,\dots,\mathcal{J}_d)
=
\sum_{m=0}^k
(-1)^{k-m}
\biggl[\begin{matrix} d-m \\ d-k \end{matrix}\biggr]
\sum_{\substack{\sigma\in \mathfrak{S}_d \\ \ell(\sigma)=m}} \sigma,
\]
\[
h_k(\mathcal{J}_1,\dots,\mathcal{J}_d)
=
\sum_{m=0}^k
\biggl\{\begin{matrix} d-m \\ d-k \end{matrix}\biggr\}
\sum_{\substack{\sigma\in \mathfrak{S}_d \\ \ell(\sigma)=m}} \sigma,
\]
where $\ell(\sigma)$ denotes the number of cycles of $\sigma$.

Equivalently, at the level of generating series one has
\[
\prod_{i=1}^d (1+z \mathcal{J}_i)
=
\sum_{k\ge 0} z^k e_k(\mathcal{J}_1,\dots,\mathcal{J}_d),
\qquad
\prod_{i=1}^d \frac{1}{1-z \mathcal{J}_i}
=
\sum_{k\ge 0} z^k h_k(\mathcal{J}_1, \dots, \mathcal{J}_d),
\]
which makes explicit the appearance of Stirling numbers in the expansion of
symmetric functions of Jucys--Murphy elements.

\subsection{Evaluation on natural numbers}

A useful classical parallel is obtained by evaluating elementary and complete
homogeneous symmetric polynomials on the alphabet of natural numbers
$\{1,2,\dots,n\}$.
One has the identities
\[
e_k(1,2,\dots,n)
=
\biggl[\begin{matrix} n+1 \\ n+1-k \end{matrix}\biggr],
\qquad
h_k(1,2,\dots,n)
=
\biggl\{\begin{matrix} n+k \\ n \end{matrix}\biggr\},
\]
which provide another realization of Stirling numbers of the first and second kind.

Equivalently, the generating series
\[
\prod_{i=1}^n (1+i z)
=
\sum_{k\ge 0}
\biggl[\begin{matrix} n+1 \\ n+1-k \end{matrix}\biggr] z^k,
\qquad
\prod_{i=1}^n \frac{1}{1-i z}
=
\sum_{k\ge 0}
\biggl\{\begin{matrix} n+k \\ n \end{matrix}\biggr\} z^k
\]
make explicit the complete analogy between evaluation on Jucys--Murphy elements
and on the natural numbers

\subsection{Shifted Frobenius coordinates}
\label{sec:Shifted:Frob}

Frobenius coordinates and their shifted counterpart provide a very convenient way to deal with partitions. 
\begin{definition}
The Frobenius coordinates of a partition $\lambda = (\lambda_1, \lambda_2, \lambda_3, \dots)$ is the $(2R+1)$-uple of non-negative integers
$$
\biggl(R; a_1, \dots, a_R; b_1, \dots, b_R\biggr), \qquad a_1 > a_2 > \dots a_R \geq 0, \qquad 0 \leq b_1 < b_2 < \dots b_R,
$$
obtained in the following way.
(When we want to emphasize the dependence of these quantities on $\lambda$, we write $R(\lambda)$, $a_i(\lambda)$, $b_i(\lambda)$.)
First, $R$ is the number of boxes in the principal diagonal of the Young diagram  $Y_{\lambda}$ (namely, the boxes of the form $(i,i)$) and let $x_1, \dots, x_R$ the boxes of the diagonal. Let $a_i$ be the arm-length of $x_i$ and let $b_i$ be the leg-length of $x_i$, namely, denoting $\lambda^T$ the partition corresponding to the transposed Young diagram of $\lambda$, we have
\begin{equation}
    a_i = \lambda_i - i \quad \text{ and } \quad b_i = \lambda_i^T - i, \qquad i = 1, 2, \dots, R(\lambda).
\end{equation}

A representation of the Frobenius coordinates is depicted in the figure below.
\begin{example}
Consider \(\lambda=(6,5,3,1)\). Then $R(\lambda) = 3$ and the shifted Frobenius coordinates are 
\begin{equation}
a_1' = 11/2, \; a_2' = 7/2, \; a_3' = 1/2, \qquad b_1' = 7/2, \; b_2' = 3/2, \; b_3' = 1/2
\end{equation}.
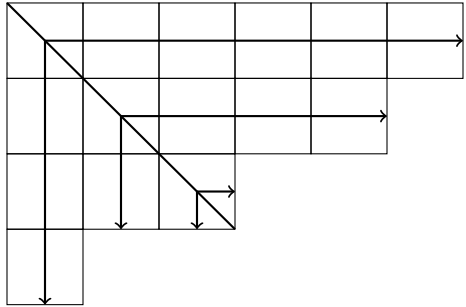
\[
\YoungDiagramFrob{6,5,3,1}
\]
\end{example}

\end{definition}

Note that the partition condition $\lambda_1 \geq \lambda_2 \geq \dots $ implies the inequalities between the $a_i$ and between the $b_i$ above. Moreover note that

\begin{equation}\label{eq:Frobcoordsum}
|\lambda| = \sum_{i}\lambda_i=  R + \sum_{i=1}^R a_i + b_i.
\end{equation}
On the other hand, given such $(2R+1)$-uple of integers, one reconstructs a Young tableau uniquely. Therefore Frobenius coordinates give a bijection between the set of all partitions and the set of all such tuples. One could reabsorb the summand $R$ in \ref{eq:Frobcoordsum} by shifting each of the coordinates by one half. The shifted Frobenius coordinates
$$
(R; a_1', \dots, a_R'; b_1', \dots, b_R')
$$
are simply defined from the standard Frobenius ones by 
\begin{equation}
    a_i' = a_i + \frac{1}{2}, \qquad b_i' = b_i + \frac{1}{2},
\end{equation} 
so that
$$
|\lambda| = \sum_{i=1}^R a_i' + b_i'.
$$

\end{document}